\documentclass{amsart}
\theoremstyle{plain}
\newtheorem{theorem}{Theorem}[section]
\newtheorem{lemma}[theorem]{Lemma}

\theoremstyle{definition}
\newtheorem{definition}[theorem]{Definition}

\newtheorem{corollary}[theorem]{Corollary}
\newtheorem{proposition}[theorem]{Proposition}

\usepackage{amsmath,amssymb}
\usepackage{hyperref}
\usepackage{tikz}
\usepackage{mathtools}
\usepackage{tikz-cd}
\usepackage[english]{babel}
\usepackage{mathrsfs,enumerate}

\newtheorem{case}{Case}

\newcommand{\KK}{\mathbb{K}} % general field
\newcommand{\RR}{\mathbb{R}} % real
\newcommand{\CC}{\mathbb{C}} % complex
\newcommand{\PP}{\mathbb{P}} % projective complex

\DeclareMathOperator{\rank}{rank}
\DeclareMathOperator{\Ima}{Im}
\DeclareMathOperator{\Seg}{Seg}
\DeclareMathOperator{\spa}{span}
\DeclareMathOperator{\Var}{Var}
\DeclareMathOperator{\dist}{dist}
\DeclareMathOperator{\cl}{cl}

\DeclarePairedDelimiter\floor{\lfloor}{\rfloor}
\DeclarePairedDelimiter\norm{\lVert}{\rVert}

\numberwithin{equation}{section}

\title{Semialgebraic Geometry of Nonnegative Tensor Rank}
\thanks{YQ and PC are supported by the ERC under the European Community's Seventh Framework Program FP7/2007-2013 Grant 320594. 
LHL is supported by AFOSR FA9550-13-1-0133, DARPA D15AP00109, NSF IIS 1546413, DMS 1209136, and DMS 1057064.}

\author{Yang~Qi}
\author{Pierre~Comon}
\address{CNRS, Gipsa-Lab,  Universit\'{e} Grenoble Alpes, F-38000 Grenoble, France}
\email{yang.qi@gipsa-lab.grenoble-inp.fr, pierre.comon@gipsa-lab.grenoble-inp.fr}
\author{Lek-Heng~Lim}
\address{Computational and Applied Mathematics Initiative, Department of Statistics, University of Chicago, 5734 South University Avenue, Chicago, IL 60637, USA}
\email{lekheng@galton.uchicago.edu}

\begin{document}

\begin{abstract}
We study the semialgebraic structure of $D_r$, the set of nonnegative tensors of nonnegative rank not more than $r$, and use the results to infer various properties of nonnegative tensor rank. We  determine all nonnegative typical ranks for cubical nonnegative tensors and show that the direct sum conjecture is true for nonnegative tensor rank. We show that nonnegative, real, and complex ranks are all equal for a general nonnegative tensor of nonnegative rank strictly less than the complex generic rank. In addition, such nonnegative tensors always have unique nonnegative rank-$r$ decompositions if the real tensor space is $r$-identifiable. We determine conditions under which a best nonnegative rank-$r$ approximation has a unique nonnegative rank-$r$ decomposition: for $r \le 3$, this is always the case; for general $r$, this is the case when the best nonnegative rank-$r$ approximation does not lie on the boundary of $D_r$.  Many of our general identifiability results also apply to real tensors and real symmetric tensors. 
\end{abstract}

\keywords{nonnegative tensors, nonnegative tensor rank, nonnegative typical ranks, real tensor rank, symmetric tensor rank, best nonnegative rank-$r$ approximations, semialgebraic geometry, uniqueness and identifiability}

\subjclass[2010]{14P10, 15A69, 41A50, 41A52}

\maketitle

\section{Introduction}

In many applications, notably algebraic statistics \cite{FriedlandGross12:ja, Friedland13:laa, AllmanRhodes08:aam, AllmanRhodes03:mb, PachterSturmfels05, DrtonSturmfelsSullivant09, AllRhoSturmZw15:laa}, one frequently needs to find (i) the nonnegative rank, (ii) a nonnegative rank-$r$ decomposition, or (iii) a best nonnegative rank-$r$ approximation, of a nonnegative third order tensor. Such problems also arise for instance in chemometrics \cite{LimC09:jchemo} and hyperspectral imaging \cite{VegaCCCC16:tgrs}, where quantities like concentration and intensity can only take on nonnegative values. This article addresses questions pertaining to these three problems using tools from semialgebraic geometry.

Questions regarding nonnegative decompositions of a nonnegative tensor are often regarded as being more difficult than the corresponding questions over the complex numbers. One reason is that the tools of classical algebraic geometry are often at one's disposal in the latter case but not the former. In this article we study nonnegative tensors under the light of semialgebraic geometry. The first main result of our article (cf.\ Theorem~\ref{thm:riden1}) is that for a general nonnegative tensor with nonnegative rank strictly less than the complex generic rank, its rank over complex numbers, real numbers, and nonnegative real numbers, are all equal. Furthermore, for such a nonnegative tensor, its nonnegative rank-$r$ decomposition is unique if the real tensor space is $r$-identifiable. We determine the nonnegative typical ranks in Propositions~\ref{prop:typical1} and \ref{prop:typical2} and show in Lemma~\ref{lem:directsum} that the nonnegative direct sum conjecture is true, i.e., the nonnegative rank of the direct sum of two nonnegative tensors equals the sum of the respective nonnegative ranks. In our earlier work \cite{QComonLim14:arxiv}, we showed that a general nonnegative tensor has a unique best nonnegative rank-$r$ \emph{approximation}. But it  remains to be seen whether this approximation itself has a unique nonnegative rank-$r$ \emph{decomposition}; we show that this is the case for $r \le 3$ in Theorem~\ref{thm:uni23}, and, for general $r$, we show in  Corollary~\ref{cor:uniqueness} that uniqueness holds for an open subset of nonnegative tensors under some conditions on the tensor space.

The paper is organized as follows. Section~\ref{sec:semialgebraic} lists some preliminary facts in semialgebraic geometry. The definition of $X$-rank and its basic properties are introduced in Section~\ref{sec:r-ranks}. Lemma~\ref{le:interval} is necessary to determine nonnegative typical ranks in Propositions~\ref{prop:typical1} and \ref{prop:typical2}. Our main contributions are then presented in Sections~\ref{sec:nonnegative}, \ref{sec:typmax}, \ref{sec:generic}.   Although we focus on nonnegative tensors, some of our techniques apply almost verbatim to real tensors and real symmetric tensors, and thus we will also derive a few identifiability results for such tensors. 

We begin with a short list of standard definitions. 
Let $V_1, \dots, V_d$ be vector spaces over a field $\KK$, and denote the dual of $V_i$ by $V_i^*$. The tensor space $V_1^* \otimes \cdots \otimes V_d^*$ is the space of multilinear $\KK$-valued functions on $V_1 \times \cdots \times V_d$. Its elements are called order-$d$ tensors or $d$-tensors or just tensors if the order is implicit. We will write $\KK^{n_1 \times \dots \times n_d} = \KK^{n_1} \otimes \dots \otimes \KK^{n_d}$ and regard the elements as $d$-dimensional hypermatrices.

A nonzero tensor in $V_1 \otimes \cdots \otimes V_d$ is said to have rank-one if it is of the form $v_1 \otimes \cdots \otimes v_d$, where $v_i \in V_i$ and $v_1 \otimes \cdots \otimes v_d$ is defined by
\[
v_1 \otimes \cdots \otimes v_d(u_1, \dots, u_d) = v_1(u_1) \cdots v_d(u_d)
\]
for all $u_i \in V_i^*$. The rank of a nonzero tensor $T$, denoted by $\rank(T)$, is the minimum number $r$ such that $T$ is a sum of $r$ rank-one tensors. In addition, $\rank(T) = 0$ iff $T =0$. An expression of $T$ as a sum of $r = \rank(T)$ rank-one tensors is called a rank-$r$ decomposition\footnote{An expression of $T$ as a sum of $s$ rank-one tensors where $s$ is not necessarily $\rank(T)$ will just be called an $s$-term decomposition.}.
A rank-$r$ decomposition
\begin{equation}\label{eq:rrd} 
T = \sum_{i=1}^r T_i,\qquad T_i=u^{(1)}_i \otimes \cdots \otimes u^{(d)}_i,
\end{equation}
is said to be (essentially) unique if the unordered set $\{T_i :i=1,\dots, r\}$ is unique \cite{Como14:spmag}, i.e., each $u^{(k)}_i$ is unique up to permutation and scaling \cite{Krus77:laa,Hack12,Land12,DomanovDeLa131:simax,HLA}.
The tensor space $V_1 \otimes \cdots \otimes V_d$ is said to be \emph{$r$-identifiable} if a general rank-$r$ tensor has a unique rank-$r$ decomposition \cite{ChiantiniOtt12:simax}.  There has been intense research on tensor ranks and uniqueness of rank-$r$ decompositions. See \cite{Como14:spmag} for a review.

 We note that the names \textsc{parafac}, \textsc{candecomp}, canonical polyadic, or \textsc{cp} decomposition have often been used in the literature for \eqref{eq:rrd}. However \eqref{eq:rrd} and the corresponding notion of rank were originally proposed by F.~L.~Hitchcock \cite{Hitch}, and  it was followed by many subsequent works in mathematics long before the psychometricians \cite{CC, Harsh} coined the names \textsc{candecomp} and \textsc{parafac}. Hitchcock had used `polyadic' in a different sense and the terms \textsc{cp}-rank and \textsc{cp} decompositions are better known as something entirely different \cite{CP1,CP2,CP3,CP4}. As such we think it is fair to use a neutral and unambiguous term like `rank-$r$ decomposition' to describe \eqref{eq:rrd}.

In this article, the field $\KK$ will be either the field of real numbers $\RR$ or complex numbers $\CC$. We will also extend the above to a semiring, denoted by $\mathcal{R}$. Of particular interest to us is the semiring of nonnegative real numbers $\RR_+\coloneqq [0,\infty)$. It is possible that $\mathcal{R} = \RR$ or $\CC$, i.e., a result stated for semiring would also apply to a field unless stated otherwise. For convenience of notations, all our results are stated for $3$-tensors, i.e., $d=3$, although most of them can be generalized to tensors of arbitrary order without difficulties.

\section{Semialgebraic geometry}\label{sec:semialgebraic}

In this section we briefly review some well-known facts in semialgebraic geometry, providing in particular a summary of the relevant portions of \cite{BCR98, Coste02:raag, Milnor68, Durfee83:tams, Coste05} for our later use.

A \emph{semialgebraic} subset of $\RR^n$ is the union of finitely many subsets of the form
\[
\{ x \in \RR^n : P(x) = 0,\; Q_1(x) > 0, \dots, Q_m(x) > 0\},
\]
where $P, Q_1, \dots, Q_m \in \RR[X_1, \dots, X_n]$,  are polynomials in $n$ variables with real coefficients. Let $S$ and $T$ be semialgebraic sets. A map $f \colon S \to T$ is called \emph{semialgebraic} if its graph $G(f) \coloneqq \{ (s, t) \in S \times T : f(s) = t \}$ is semialgebraic. A semialgebraic set is called \emph{nonsingular} if it is an open subset of the set of nonsingular points of some algebraic set. A \emph{Nash manifold} is a  semialgebraic analytic submanifold of  $\RR^n$ and a \emph{Nash mapping} between Nash manifolds is an analytic mapping with a semialgebraic graph.

A  point $p$ in a semialgebraic set $S$ is said to be \emph{general} with respect to some property $\mathscr{P}$ if the points in $S$ that do not have the property $\mathscr{P}$ are all contained in a semialgebraic subset $C$ of $S$ with $\dim C < \dim S$ and $p \notin C$. To aid readers unacquainted with the notion, we give familiar measure theoretic and topological interpretations of a general point but note that these cannot replace its formal definition. Given the Lebesgue measure $\mu$ on $S$, if a point $p \in S$ is general with respect to a property $\mathscr{P}$, then (i) $C \coloneqq \{q\in S \colon q \;\text{does not satisfy}\; \mathscr{P}\}$ is a measure-zero subset of $S$; and (ii) $p \notin C$. Hence in the sense of measure theory, the statement that a general point satisfies $\mathscr{P}$ is equivalent to the statement that almost every point satisfies $\mathscr{P}$. On the other hand, in the sense of topology, the statement that a general point satisfies $\mathscr{P}$ has a stronger connotation --- it implies that the subset $C$ lies in a hypersurface of $S$. Take $S = \RR$ for example, that a general point satisfies $\mathscr{P}$ implies that at most finitely many points in $\RR$ do not satisfy $\mathscr{P}$. Note that this is a stronger conclusion than `almost every point in $S$ satisfies $\mathscr{P}$' in the measure theoretic sense.

Let $f \colon M \to N$ be a Nash mapping between Nash manifolds $M$ and $N$. The usual semialgebraic version \label{SardThm} of Sard's theorem \cite{BCR98} says that the set of critical \emph{values} of $f$ is a semialgebraic subset of $N$ with smaller dimension. As we focus on polynomial maps in this article, we have the following stronger version of Sard's theorem about critical \emph{points} of $f$. 

\begin{lemma}\label{thm:sard}
 Let $f \colon \RR^m \to \RR^n$ be a nonconstant polynomial map. Then the set of critical points of $f$ is a subvariety of $\RR^m$, with dimension strictly less than $m$. 
\end{lemma}
\begin{proof}
Let $d \coloneqq \dim \Ima f$ and $\nabla f$ be the Jacobian of $f$ (i.e., the matrix of first order partial derivatives if we choose coordinates). Then every $d \times d$ minor of $\nabla f$ must vanish on the points $x \in \RR^m$ where $\nabla f(x)$ has rank strictly less than $d$. At least one of these minors is not  identically zero  since there are points $x \in \RR^m$ where $\nabla f(x)$ has rank exactly $d$ . Thus these minors define a subvariety whose dimension is strictly less than $m$.
\end{proof}
Aside from Sard's theorem, we also quote a few selected results and definitions from \cite{BCR98, Durfee83:tams} for the reader's easy reference. These results are somewhat technical and although they logically belong to this section, we will not need them until Section~\ref{sec:generic}. In particular,  Sections~\ref{sec:r-ranks} through \ref{sec:typmax} do not require any of the following.
\begin{theorem}[Nash Tubular Neighborhood]\label{thm:tubular}
Let $N \subset \RR^n$ be a Nash submanifold. Then there is an open semialgebraic neighborhood $U \subset \RR^n$ and a Nash retraction $f \colon U \to N$ such that $\operatorname{dist}(p, N) = \norm{p - f(p)}$ for each $p \in U$. Here $\lVert\, \cdot\, \rVert$ denotes the Euclidean norm in $\RR^n$.
\end{theorem}

%\begin{theorem}[Curve Selection Lemma]\label{thm:curvesl}
%Let $S \subset \RR^n$ be a semialgebraic subset and let $p \in\overline{S}$, the Euclidean closure of $S$. Then there is a Nash curve $c \colon (-1, 1) \to \RR^n$ such that $c(0) = p$ and $c((0, 1)) \subset S$.
%\end{theorem}

%The following definition of Whitney stratification is based on  \cite{DrusvyatskiyLewis13:mp} and \cite{BCR98}, adapted slightly for our own use.
\begin{definition}
A \emph{Whitney stratification} of a semialgebraic set $S \subseteq \RR^n$ is a finite partition of $S$ into semialgebraically connected submanifolds $S = \bigcup_i S_i$ satisfying the following two conditions, known respectively as the `frontier condition' and `Whitney condition (a)'.
\begin{enumerate}[\upshape (i)]
\item For $i \ne j$, if $S_i \cap \cl(S_j) \ne \varnothing$, then $S_i \subseteq \cl(S_j) \setminus S_j$.
\item For any sequence of points $(x_k)$ in a stratum $S_j$, if $x_k$ converges to a point $y$ in a stratum $S_i$, and the sequence of tangent $(\dim S_j)$-planes $\mathsf{T}_{x_k} S_j$ converges to a $(\dim S_j)$-plane $T$, then $T$ contains the tangent $(\dim S_i)$-plane $\mathsf{T}_{y} S_i$.
\end{enumerate}
\end{definition}

Given two finite families $\{ B_i \}$ and $\{ C_j \}$ of subsets of $\RR^n$, $\{ B_i \}$ is said to be compatible with $\{ C_j \}$ if $B_i \cap C_j = \varnothing$ or $B_i \subseteq C_j$ for all $i$ and $j$.

\begin{theorem}\label{thm:compat}
For semialgebraic subsets $S, C_1, \dots, C_m$ of $\RR^n$, $S$ admits a Whitney stratification compatible with $C_1, \dots, C_m$.
\end{theorem}

\begin{proposition}
Let $f \colon S \to \RR^n$ be a semialgebraic function on a semialgebraic set. Then $S$ admits a Whitney stratification $S = \bigcup_i S_i$ such that each graph of $f \vert_{S_i}$ is a nonsingular semialgebraic set.
\end{proposition}

\begin{proposition}\label{prop:smooth}
Let $S$ be a nonsingular semialgebraic set, and $f \colon S \to \RR^n$ be a function such that $G(f)$ is nonsingular and semialgebraic. Then the set of points of $S$ where $f$ is not differentiable is contained in a closed lower-dimensional semialgebraic subset of $S$.
\end{proposition}

%We round out our list of background material in semialgebraic geometry with a result that, despite its name, is very useful \cite[Theorem~9.6.4]{BCR98}.
%\begin{theorem}[Hardt's Triviality]\label{thm:hardttrivial}
%Let $S$ and $T$ be semialgebraic sets, and $f \colon S \to T$ be a continuous semialgebraic map. Given finitely many semialgebraic subsets $S_1, \dots, S_m$ of $S$. There is a finite semialgebraic partition $T = \bigcup_j T_j$ and a semialgebraic homeomorphism $\eta_j \colon T_j \times F_j \to f^{-1}(T_j)$ such that $f \circ \eta_j$ is the projection $T_j \times F_j \to T_j$ for each $j$, compatible with $S_1, \dots, S_m$.
%\end{theorem}

\section{$X$-ranks}\label{sec:r-ranks}

There has been several attempts to describe tensor ranks in different settings in a unified and general way, e.g.\ \cite{BlekTei14:ma, Teitler14:arxiv} but they do not usually include nonnegative rank as a special case. Here we introduce a generalization of $X$-rank \cite{Zak2004}  to the setting of an arbitrary cone $X$ and coefficients in a semiring $\mathcal{R}$ in order to treat nonnegative, real, and complex tensor ranks in a unified setting.

\begin{definition}
Let $\KK$ be a field, and $\mathcal{R} \subseteq \KK$ be a semiring. Given a vector space $V$ over $\KK$, and a subset $X \subseteq V$, an \emph{$\mathcal{R}$-span} of $X$, denoted by $\spa_\mathcal{R} (X)$, is the set of all finite $\mathcal{R}$-linear combinations of elements of $X$, that is,
\[
\spa_\mathcal{R} (X) \coloneqq \left\{ \sum_{i=1}^k \alpha_i x_i : k > 0, \; \alpha_i \in \mathcal{R}, \; x_i \in X \right\}.
\]
\end{definition}

When $\mathcal{R} = \KK$, an $\mathcal{R}$-span is a subspace. When $\KK = \RR$ and $\mathcal{R} = \RR_+$, an $\mathcal{R}$-span is a convex cone. 
We will denote the $\mathbb{R}_+$-cone of nonnegative vectors in a vector space $V$ by either\footnote{Allowing both superscript and subscript provides notational flexibility when indices or powers are involved.} $V^+$ or $V_+$. Note that in order to specify $V_+$, we will need to first specify a choice of basis on $V$. See \cite{QComonLim14:arxiv} for further discussions. With this notation, $V_1^+ \otimes \dots \otimes V_d^+$ is the cone of nonnegative tensors as defined in \cite[Definition~2]{QComonLim14:arxiv}.

\begin{definition}\label{def:cone}
We say $X$ is an \emph{$\mathcal{R}$-cone}, if for $x \in X$ we always have $\lambda x \in X$ for any $\lambda \in \mathcal{R}$. Given an $\mathcal{R}$-cone $X$, for any $p \in \spa_\mathcal{R} (X)$, the \emph{$X$-rank} of $p$, $\rank_X(p)$, is defined to be 
\[
\rank_X(p) \coloneqq \min \{r : p = x_1 + \cdots + x_r; \; x_1, \dots, x_r \in X \}.
\]
\end{definition}
Recall that in algebraic geometry, the \emph{affine cone} $X \subseteq \mathbb{K}^n$ over a projective variety $Y \subseteq \mathbb{KP}^{n-1}$ is defined as $X \coloneqq \pi^{-1} (Y) \cup \{0\}$ where $\pi \colon \mathbb{K}^n \setminus \{0\} \to \mathbb{KP}^{n-1}$, $(x_1, \dots,x_n) \mapsto [x_1 : \dots : x_n]$ is the canonical projection. Note that an affine cone is a $\KK$-cone in the sense of Definition~\ref{def:cone}.
\begin{enumerate}[\upshape (i)]
\item Let $\mathcal{R} = \KK = \RR$, $V = V_1\otimes \dots \otimes V_d$, and $X$ be the cone of tensors of rank $\le 1$ (i.e., affine cone over the real projective Segre variety). Then $\rank_X(p)$ is the real rank of $p$, usually denoted $\rank_{\RR}(p)$. Real tensor rank is invariant under the action of $\operatorname{GL}(V_1) \times \cdots \times \operatorname{GL}(V_d)$, where $\operatorname{GL}(V)$ denotes the general linear group of $V$.

\item Let $\mathcal{R} = \RR_+$, $\KK = \RR$, $V = V_1\otimes \dots \otimes V_d$, and $X$ be the $\RR_+$-cone of nonnegative tensors of rank $\le 1$. Then $\rank_X(p)$ is the nonnegative rank of $p$, usually denoted $\rank_+(p)$. Nonnegative tensor rank is invariant under the action of
\[
\{(g_1, \dots, g_d) \in \operatorname{GL}(V_1) \times \cdots \times \operatorname{GL}(V_d) \colon g_i(V_i^+) \subseteq V_i^+, \; i = 1, \dots, d \}.
\]
Note that this set is just a monoid --- it does not necessarily contain the inverses of its elements.

\item Let $\mathcal{R} = \KK$ be an algebraically closed field and $X$ be the affine cone over an irreducible nondegenerate projective variety. Then $\rank_X(p)$ is the $X$-rank as defined in \cite{Zak2004, Land12, BlekTei14:ma}. $X$-rank is invariant under the automorphism group of $X$, a subgroup of $\operatorname{GL}(V)$.
\end{enumerate}

The discussions above are purely algebraic but subsequent discussions will require topological structures on our vector space and field. Recall that a topological vector space over a topological field is one where the vector addition and scalar multiplication are continuous. We will not require any results regarding topological vector space beyond its definition.
\begin{definition}\label{def:typical}
Let $V$ be a finite-dimensional topological vector space over a topological field $\mathbb{K}$ of characteristic zero, and $\mathcal{R} \subseteq \mathbb{K}$ be a semiring.  Let $X \subseteq V$ be an $\mathcal{R}$-cone such that $\spa_\mathcal{R} (X)$ contains a nonempty open subset of $V$. If the set $\{ p \in \spa_\mathcal{R} (X) : \rank_X(p) = r \}$ contains a nonempty open subset of $V$, then $r$ is called a \emph{typical} $X$-rank. In particular, when $\KK = \CC$ and $V$ is endowed with the Zariski topology, $r$ is called a \emph{complex generic} $X$-rank whenever $\{ p \in \operatorname{span}_{\CC} (X) : \rank_X(p) = r \}$ contains a nonempty Zariski open subset of $V$. The \emph{maximum typical $X$-rank} is
\[
\max \{ r : r \; \text{is a typical $X$-rank of} \;\spa_\mathcal{R} (X) \},
\]
whereas the \emph{maximum $X$-rank} is
\[
\max \{ \rank_X(p) :  p \in \spa_\mathcal{R} (X) \}.
\]
\end{definition}

\label{measureexp}To provide a more familiar perspective, when $\mathbb{K} = \RR$ or $\CC$ and $V$ is endowed with the Euclidean topology and the Lebesgue measure, then $r$ is a typical $X$-rank whenever $\{ p \in \spa_\mathcal{R} (X) : \rank_X(p) = r \}$ has positive measure.

 Recall that a variety is called \emph{irreducible} if it is not the union of two nonempty proper subvarieties. If the ideal of an affine variety $X \subseteq \CC^n$ is generated by  polynomials with real coefficients $f_1, \dots, f_k$, we will denote by $X(\RR)$ the set of \emph{real points} of $X$, i.e., $X(\RR) = X \cap \RR^n$. In fact $X(\RR)$ equals the zero locus of $f_1, \dots, f_k$ in $\RR^n$. On the other hand, if $Y \subseteq \RR^n$ is a real variety defined by real polynomials $f_1, \dots, f_k$, we will denote by $Y(\CC)$ the \emph{complexification} of $Y$, the complex variety  defined by $f_1, \dots, f_k$ in $\CC^n$. For an irreducible real affine variety  $Y \subseteq \RR^n$, its complexification $Y(\CC)$ is also irreducible \cite{BlekTei14:ma}. Furthermore $Y$ is Zariski dense in $Y(\CC)$ if and only if $Y(\CC)$ has a nonsingular real point \cite{BlekTei14:ma, Sottile16:arxiv}.  

A (projective) variety $X \subseteq V$ ($X \subseteq \PP V$) is said to be \emph{nondegenerate} if $X$ is not contained in any hyperplane. It is shown in \cite[Theorem~2]{BlekTei14:ma} that when $X$ is an irreducible nondegenerate real projective variety whose complexification $X(\CC)$ has a real smooth point, there is a unique complex generic $X$-rank, and it is equal to the minimum real typical $X$-rank. For example, the space of $2 \times 2 \times 2$ tensors has the complex generic rank $2$ and the real typical ranks $2$ and $3$ \cite{DesiL08:simax}.

We deduce the following lemma using an argument in \cite{Friedl12:laa}, where it is proved for the case $\KK = \RR$, $V = V_1 \otimes V_2 \otimes V_3$, and $X =\{ A \in V : \rank_{\RR}(A) \le 1\}$. See also \cite[Theorem~1.1]{BernardiBlekhermanOttaviani15:arXiv} for the case where $X$ is the affine cone of a nondegenerate irreducible real projective variety.
\begin{lemma}\label{le:interval}
Let $\KK = \RR$ and $X$ be a nonempty semialgebraic $\mathcal{R}$-cone  whose Zariski closure $\overline{X}$ is a nondegenerate irreducible real variety that is Zariski dense in $\overline{X}(\CC)$.  If $m$ and $M$ are two typical $X$-ranks, then any integer between $m$ and $M$ is also a typical $X$-rank.
\end{lemma}

\begin{proof}
 Let $\dim V = n$.  For each $k \in \mathbb{N}$, define the polynomial map $\varphi_k$ by
\[
  \varphi_k \colon X \times \dots \times X \to \spa_\mathcal{R} (X),\quad
  (x_1, \dots, x_k) \mapsto x_1 + \cdots + x_k.
\]
Assume \textsc{wlog} that $m \le M$ and suppose that $r \in \{ m, \dots, M \}$ is the minimum integer which is not a typical $X$-rank. For any fixed $k \in \mathbb{N}$ and for any open subset $\mathcal{W} \subseteq V$, $\varphi_k^{-1} (\mathcal{W})$ is open in $X \times \dots \times X$; thus  it is a union of open subsets of the form $\mathcal{U}_1 \times \cdots \times \mathcal{U}_k$ where each $\mathcal{U}_i$ is open in $X$.  Since  $\overline{X}$ is irreducible, the dimension of each $\mathcal{U}_i$ equals $\dim X$.  By \cite[Exercise~II.3.22]{Hart77}, the dimension of each $\varphi_r (\mathcal{U}_1 \times \cdots \times \mathcal{U}_r)$ equals $n$. So every nonempty open subset of  $\Ima{\varphi_r}$  has dimension $n$.  Since $r$ is not a typical rank, $\Ima{\varphi_r} \setminus \Ima{\varphi_{r-1}}$ does not contain a subset of dimension $n$, and thus  $\Ima{\varphi_r} \setminus \Ima{\varphi_{r-1}}$ does not contain an open subset of $\Ima{\varphi_r}$, which implies that a general  $p = x_1 + \cdots + x_r \in \Ima{\varphi_r}$ is  within $\Ima{\varphi_{r-1}}$, i.e., $p = \widetilde{x}_1 + \cdots + \widetilde{x}_{r-1}$. Hence a general $q = x_1 + \cdots + x_{r+1} \in \Ima{\varphi_{r+1}}$ can be written with $r$ summands as $q = \widetilde{x}_1 + \cdots + \widetilde{x}_{r-1} + x_{r+1}$, which is in $\Ima{\varphi_r}$. But we may repeat the same argument to conclude that $q$ is in $\Ima{\varphi_{r-1}}$. So by induction, a general point in $\Ima{\varphi_M}$ is in $\Ima{\varphi_{r-1}}$, i.e., $\dim \Ima{\varphi_M} \setminus \Ima{\varphi_{r-1}} < \dim V$, contradicting our assumption that $M$ is a typical $X$-rank.
\end{proof}

We will require the use of Lemma~\ref{le:interval} in Propositions~\ref{prop:typical1} and \ref{prop:typical2}. This simple lemma is surprisingly potent. As an illustration we provide a short proof for the main result in \cite{Blek13:fcm} (see also  \cite{BernardiBlekhermanOttaviani15:arXiv}), that every integer between $\lfloor (d+2)/2 \rfloor$ and $d$ is a typical rank of $\mathsf{S}^d(\RR^2)$, originally conjectured in \cite{ComOtt12:lmla}.
%Our short proof came of course from the benefit of hindsights that were not available to the author of \cite{Blek13:fcm}.

\begin{corollary}[Blekherman]
Every $m$ with $\lfloor (d+2)/2 \rfloor \le m \le d$ is a typical rank of $\mathsf{S}^d(\RR^2)$.
\end{corollary}

\begin{proof}
The complex generic rank $\lfloor (d+2)/2 \rfloor$ is necessarily the minimum typical rank by \cite{BlekTei14:ma}. It has been shown in \cite{CauRe11:ampa} that $f \in \mathsf{S}^d(\RR^2)$ has real rank $d$ if and only if $f$ has $d$ distinct real roots when regarded as a degree-$d$ homogeneous polynomial in two variables. Since  $d$ is the maximum real rank \cite{ComOtt12:lmla}, and  having $d$ distinct real roots imposes an open condition on $\mathsf{S}^d(\RR^2)$, $d$ is therefore the maximum typical rank. The required result then follows from Lemma~\ref{le:interval}.
\end{proof}

We now introduce a `semialgebraic version' of Terracini's lemma. First observe that for semialgebraic sets $X, Y \subseteq V$, if we define the semialgebraic map $\varphi$ by
\[
  \varphi \colon X \times Y  \to V ,\quad
  (x, y)  \mapsto x + y,
\]
then $\Ima (\varphi)$ is semialgebraic by the Tarski--Seidenberg Theorem.

\begin{lemma}[Semialgebraic Terracini's lemma]\label{lem:semiterracinilem}
Let $X$ and $Y$ be nonempty semialgebraic subsets.  Suppose their Zariski closures $\overline{X}$, $\overline{Y}$ are irreducible real varieties and that $\overline{X}(\CC)$, $\overline{Y}(\CC)$ have real smooth points. Then for  general points $x \in X$ and $y \in Y$, the tangent space of $\varphi(X \times Y)$ at $x+y$ is the span of the tangent spaces $\mathsf{T}_x X$ and $\mathsf{T}_y Y$, i.e.,
\[
\mathsf{T}_{x + y} \varphi(X \times Y) = \spa \{\mathsf{T}_{x} X, \mathsf{T}_{y} Y \}.
\]
\end{lemma}

\begin{proof}
 Since $\overline{X}$ and $\overline{Y}$ are irreducible and have real smooth points, $\overline{\varphi(X \times Y)}$ is irreducible and its complexification $\overline{\varphi(X \times Y)}(\CC)$ has real smooth points. Thus the set of smooth points of $\varphi(X \times Y)$ is open dense in $\varphi(X \times Y)$. Then for a general $(x, y) \in X \times Y$, $\varphi(x, y) = x+y$ is smooth in $\varphi(X \times Y)$. Hence 
\begin{align*}
\mathsf{T}_{x + y} \varphi(X \times Y) = \varphi_*(\mathsf{T}_{(x,y)} X \times Y) &= \varphi_* (\mathsf{T}_x X \oplus \mathsf{T}_y Y) \\
&= \mathsf{T}_x X + \mathsf{T}_y Y = \spa \{ \mathsf{T}_{x} X, \mathsf{T}_{y} Y \}.
\end{align*}
\end{proof}

The following is also immediate from Tarski--Seidenberg Theorem and our earlier work.
\begin{proposition}\label{prop:semialg}
$D_r \coloneqq \{A \in \mathbb{R}_+^{n_1 \times \dots \times n_d}  : \rank_+(A) \le r\}$ is a closed semialgebraic set, i.e., there exists a finite number of polynomials $P_1,\dots,P_m $ with real coefficients that cuts out $D_r$ as a set, i.e.,
\[
D_r = \{ A \in  \mathbb{R}^{n_1 \times \dots \times n_d} : P_1(A) \ge 0, \dots, P_m(A) \ge 0\}.
\]
Furthermore, $C_r \coloneqq \{A \in \mathbb{R}_+^{n_1 \times \dots \times n_d}  : \rank_+(A) = r\}$ is also a semialgebraic set but  not closed in general.  
\end{proposition}
\begin{proof}
By the Tarski--Seidenberg Theorem \cite{BCR98}, $D_r$ is a semialgebraic set and thus so is $C_r = D_r \setminus D_{r-1}$.
By \cite[Proposition~6.2]{LimC09:jchemo}, $D_r$ is closed. 
\end{proof}

\section{Direct sum conjecture for nonnegative rank}

We now show that the \emph{direct sum conjecture} is true for nonnegative rank.
Given vector spaces $V_1, \dots, V_d$, and $W_1, \dots, W_d$ over $\KK$, for any $A \in V_1 \otimes \cdots \otimes V_d$ and $B \in W_1 \otimes \cdots \otimes W_d$,  we have the direct sum $A \oplus B \in (V_1 \oplus W_1) \otimes \cdots \otimes (V_d \oplus W_d)$. For $d = 2$, it is obvious that the rank of a block diagonal matrix is the sum of the ranks of the diagonal blocks, i.e., if $A$ and $B$ are matrices, then
\[
\rank(A \oplus B) = \rank\left( \begin{bmatrix}A  & 0\\ 0 & B\end{bmatrix}\right) = \rank(A) + \rank(B).
\]
It has been conjectured by Strassen \cite{Strassen73:jfdruam} that the same is true for $d > 2$, i.e., $\rank(A \oplus B) = \rank(A) + \rank(B)$ for any $d$-tensors. This has been a long-standing open problem in algebraic computational complexity. We show here that the analogous statement  for nonnegative rank is true. The next two results are true for nonnegative tensors of arbitrary order $d$ but we will state and prove them for $d = 3$ for notational simplicity.

In the following, let $U_1$, $V_1$, $W_1$, $U_2$, $V_2$,$W_2$ be real vector spaces of dimensions $m_1$, $n_1$, $p_1$, $m_2$, $n_2$, $p_2$ respectively. Fix a basis for each vector space and choose the bases for $U_1 \oplus U_2$, $V_1 \oplus V_2$, and $W_1 \oplus W_2$ so that for $a = (a_1, \dots, a_{m_1}) \in U_1$ and $b = (b_1, \dots, b_{m_2}) \in U_2$, $a \oplus b$ has coordinates $a \oplus b = (a_1, \dots, a_{m_1}, b_1, \dots, b_{m_2})$ in $U_1 \oplus U_2$; likewise for $V_1 \oplus V_2$ and $W_1 \oplus W_2$.
\begin{lemma}[Nonnegative direct sum conjecture]\label{lem:directsum}
For $A \in U_1^+ \otimes V_1^+ \otimes W_1^+$ and $B \in U_2^+ \otimes V_2^+ \otimes W_2^+$,
\[
\rank_+(A\oplus B) = \rank_+(A) + \rank_+(B).
\]
\end{lemma}

\begin{proof}
Fix a basis for each vector space and let $a_{ijk}$ and $b_{i'j'k'}$ denote the coordinates of $A$ and $B$. Note that $(A \oplus B)_{ijk} = a_{ijk}$, $(A \oplus B)_{i'j'k'} = b_{i'j'k'}$ and other terms are zero. Suppose that $r\coloneqq \rank_+(A \oplus B) < \rank_+(A) + \rank_+(B)$. Let $A\oplus B = \sum_{i=1}^r u_i \otimes v_i \otimes w_i$. Then at least one of the summands $u_i \otimes v_i \otimes w_i$ is neither in $U_1^+ \otimes V_1^+ \otimes W_1^+$ nor in $U_2^+ \otimes V_2^+ \otimes W_2^+$. So without loss of generality we may assume that $u_1 \in (U_1 \oplus U_2)^+ \setminus (U_1^+ \oplus \{0\} \cup\{0\} \oplus U_2^+)$. Thus at least one of the following indices
\[
(i,j',k),\; (i, j, k'), \; (i,j',k'), \; (i', j, k'), \;  (i', j', k), \; (i', j, k),
\]
which we denote by $(\alpha, \beta, \gamma)$, will be such that $(A \oplus B)_{\alpha \beta \gamma}$ is positive, a contradiction.
\end{proof}
We may also deduce the following, clearly also true for $d > 3$, from the above proof.
\begin{corollary}
If $A$ and $B$ have unique nonnegative rank decompositions in $U_1^+ \otimes V_1^+ \otimes W_1^+$ and $U_2^+ \otimes V_2^+ \otimes W_2^+$ respectively, then $A\oplus B$ also has a unique nonnegative rank decomposition.
\end{corollary}

For a real tensor $A \in \mathbb{R}^{m_1 \times \dots \times m_d} \subseteq \mathbb{R}^{n_1 \times \dots \times n_d} $, the real rank of $A$ regarded as a tensor in $\mathbb{R}^{m_1 \times \dots \times m_d} $ equals the real rank of $A$ regarded as a tensor in $\mathbb{R}^{n_1 \times \dots \times n_d} $  \cite[Proposition~3.1]{DesiL08:simax}. As a corollary of Lemma~\ref{lem:directsum}, we see that this also holds for nonnegative rank. 

In the following, let $U_1 \subseteq U_2$, $V_1 \subseteq V_2$, and $W_1 \subseteq W_2$ be inclusions of real vector spaces. Choose bases for $U_2$, $V_2$, and $W_2$ such that $u \in U_1$ has coordinates $u = (u_1, \dots, u_{m_1}, 0, \dots, 0)$ as a vector in $U_2$; likewise for $V_2$ and $W_2$.  Then we have the following corollary, which is stated for $d=3$, but can be easily generalized to arbitrary $d > 3$.
\begin{corollary}\label{cor:rksrkl}
Let $A \in U_1^+ \otimes V_1^+ \otimes W_1^+ \subseteq U_2^+ \otimes V_2^+ \otimes W_2^+$. Then the nonnegative rank of $A$ regarded as a nonnegative tensor in $U_1^+ \otimes V_1^+ \otimes W_1^+$ is the same as the nonnegative rank of $A$ regarded as a nonnegative tensor in $U_2^+ \otimes V_2^+ \otimes W_2^+$.
\end{corollary}
\begin{proof}
Let $U'_1 \subseteq U_2$ be a complementary subspace of $U_1$, i.e., $U_2 = U_1 \oplus U'_1$. So $u' \in U'_1$ has coordinates $u' = (0, \dots, 0, u'_{m_1+1}, \dots, u'_{m_2})$ as a vector in $U_2$. Likewise, we let $V'_1 \subseteq V_2$ and $W'_1 \subseteq W_2$ be complementary subspaces of $V_1$ and $W_1$. The required statement then follows from applying Lemma~\ref{lem:directsum} to the case $A \in U_1^+ \otimes V_1^+ \otimes W_1^+$ and $B \coloneqq 0 \in U^{\prime +}_1 \otimes V^{\prime +}_1 \otimes W^{\prime +}_1$.
\end{proof}
The following simple observation is a nonnegative analogue of \cite[Corollary~3.3]{DesiL08:simax}. We assume that we fix a basis for each $V_i$ so that $V_i^+$ is defined, $i = 1, \dots, d$.
\begin{proposition}
For any $k \in \{2,\dots,d-1\}$, let $A \in V_1^+ \otimes \cdots \otimes V_k^+$ be arbitrary  and let $u_{k+1} \in V_{k+1}^+, \dots, u_d \in V_d^+$ be nonzero. Then
\[
\rank_+ (A) = \rank_+ (A \otimes u_{k+1} \otimes \cdots \otimes u_d).
\]
\end{proposition}
\begin{proof}
The isomorphism of $\RR_+$-cones,
\[
V_1^+ \otimes \dots \otimes V_k^+ \cong V_1^+ \otimes \dots \otimes V_k^+ \otimes \operatorname{span}_{\RR_+}(u_{k+1}) \otimes \dots \otimes \operatorname{span}_{\RR_+}(u_{d}),
\]
given by $A \mapsto A \otimes u_{k+1} \otimes \cdots \otimes u_{d}$ implies the required equality.
\end{proof}

\section{General equivalence of complex, real, and nonnegative ranks}\label{sec:nonnegative}

It is well-known that a real tensor may have different real and complex ranks. Likewise a nonnegative tensor may also have different nonnegative and real ranks. 
In fact, strict inequality can also occur for the nonnegative and real ranks of a nonnegative matrix, a well-known example was provided by H.~Robbins \cite{Como14:spmag}.

For the case of $3$-tensors, two explicit examples are as follows. Let $e_1,e_2\in\mathbb{R}^2$ be the standard basis vectors, i.e., $e_1 =[1,0]^\mathsf{T}$, $e_2=[0,1]^\mathsf{T}$. Let
\begin{align}
A &= e_1 \otimes e_1 \otimes e_1 + e_2 \otimes e_2 \otimes e_1 + e_1 \otimes e_2 \otimes e_2 + e_2 \otimes e_1 \otimes e_2 , \label{eq:tdecomp0}\\
B &= e_1\otimes e_1\otimes e_1-e_1\otimes e_2 \otimes e_2+e_2\otimes e_1\otimes e_2+e_2 \otimes e_2\otimes e_1. \notag
\end{align}
Then $A \in \mathbb{R}_+^{2 \times 2 \times 2} \subseteq \mathbb{R}^{2 \times 2 \times 2} $ and $B \in\mathbb{R}^{2 \times 2 \times 2}  \subseteq \mathbb{C}^{2 \times 2 \times 2} $. We have
\begin{align*}
\operatorname{rank}_{\mathbb{C}}(A) = \operatorname{rank}_{\mathbb{R}}(A) = 2&<  4 =\operatorname{rank}_{+}(A),\\
\operatorname{rank}_{\mathbb{C}}(B) =2 &< 3 = \operatorname{rank}_{\mathbb{R}}(B). 
\end{align*}
See Section~\ref{sec:typmax} for the nonnegative, real, and complex ranks of $A$ and \cite{DesiL08:simax} for the real and complex ranks of $B$. We will show in this section that this does not happen for a general nonnegative tensor of nonnegative rank strictly less than the complex generic rank --- its nonnegative, real, and complex ranks will all be equal.

For notational simplicity we focus on $3$-tensors, although many of the  statements and proofs  in this section can be generalized without difficulty to $d$-tensors for any $d > 3$. Let $U$, $V$ and $W$ be real vector spaces of dimensions $n_U$, $n_V$ and $n_W$ respectively. Denote by $V_{\CC}$ the complexification of $V$, i.e., $V_{\CC} = V \otimes_{\RR} \CC$.

We define the polynomial map
\begin{equation}\label{eq:sigma}
\begin{aligned}
  \Sigma^{\CC}_r \colon (U_{\CC} \times V_{\CC} \times W_{\CC})^{r} &\to U_{\CC} \otimes V_{\CC} \otimes W_{\CC}, \\
  (u_1, v_1, w_1, \dots, u_r, v_r, w_r) &\mapsto \sum_{i=1}^r u_i \otimes v_i \otimes w_i,
\end{aligned}
\end{equation}
and denote the restriction of $\Sigma^{\CC}_r$ to $(U \times V \times W)^{r}$ by $\Sigma^{\RR}_r$, and the restriction to $(U_+ \times V_+ \times W_+)^{r}$ by $\Sigma^{\RR_+}_r$. We have the following commutative diagram:
\begin{equation}\label{eq:sigma2}
\begin{tikzcd}
(U_+ \times V_+ \times W_+)^{r} \arrow{r}{\Sigma^{\RR_+}_r} \arrow[hookrightarrow,swap]{d} & U_+ \otimes V_+ \otimes W_+ \arrow[hookrightarrow,swap]{d} \\
(U \times V \times W)^{r} \arrow{r}{\Sigma^{\RR}_r} \arrow[hookrightarrow,swap]{d} & U \otimes V \otimes W \arrow[hookrightarrow,swap]{d} \\
(U_{\CC} \times V_{\CC} \times W_{\CC})^{r} \arrow{r}{\Sigma^{\CC}_r} & U_{\CC} \otimes V_{\CC} \otimes W_{\CC}.
\end{tikzcd}
\end{equation}

Henceforth, we will use the following abbreviated notation when specifying an element of $(U \times V \times W)^{r}$,
\begin{equation}\label{eq:abbre}
(u_1, \dots, w_r) \coloneqq   (u_1, v_1, w_1, \dots, u_r, v_r, w_r).
\end{equation}
Then we have
\[
\Ima{\Sigma^{\RR_+}_r} = D_r \coloneqq \{A \in U_+ \otimes V_+ \otimes W_+ : \rank_+(A) \le r\}.
\]
The notation is consistent with  Proposition~\ref{prop:semialg}, which also implies that $\Ima{\Sigma^{\RR_+}_r}$ is closed. Note that  $\Ima{\Sigma^{\RR}_r}$ and $\Ima{\Sigma^{\CC}_r}$ are usually not closed.

As in Definition~\ref{def:typical}, if  $r_g$ is the complex generic rank of $U_{\CC} \otimes V_{\CC} \otimes W_{\CC}$, then the set of rank-$r_g$ tensors contains a Zariski open subset. Put in another way, the complex generic rank is the minimum $r$ such that the morphism $\Sigma^{\CC}_r$ is dominant. As we mentioned earlier, the result \cite[Theorem~2]{BlekTei14:ma} shows that the complex generic rank equals the minimum real typical rank.

The expected dimension of $\Ima{\Sigma^{\RR}_r}$ is $\min\{ r(n_U + n_V + n_W - 2), n_U n_V n_W \}$ and thus the expected complex generic rank is
\[
\biggl\lceil\frac{n_U n_V n_W}{n_U + n_V + n_W - 2}\biggr\rceil,
\]
which is at least $r_g$.

\begin{definition}\label{defi:rdef}
If $\dim (\Ima{\Sigma^{\RR}_r}) < \min\{ r(n_U + n_V + n_W - 2), n_U n_V n_W \}$, then $U \otimes V \otimes W$ is called \emph{$r$-defective} over $\RR$.
\end{definition}

The definition of defectivity over $\mathbb{C}$, i.e., identical to Definition~\ref{defi:rdef} but with $U,V,W$ being complex vector spaces, is classical in algebraic geometry  \cite{Zak93}. More generally, a complex projective variety $X$ is called $r$-defective \cite{ChiaC02:tams} if the $r$th secant variety of $X$ does not have the expected dimension. In our context this is equivalent to $\dim_{\CC} (\Ima{\Sigma^{\CC}_r}) < \min\{ r(n_U + n_V + n_W - 2), n_U n_V n_W \}$.  Note that if $U \otimes V \otimes W$ is $r$-identifiable, then $U \otimes V \otimes W$ is not $r$-defective.

\begin{lemma}\label{lem:nonnegreal}
Let $r < r_g$. Then a general $A \in D_r$ has real rank $r$.
\end{lemma}
\begin{proof}
Let the Jacobian of $\Sigma^{\RR}_r$ be $\nabla \Sigma^{\RR}_r$. If $\rank(\nabla \Sigma^{\RR}_{r-1}) = \rank(\nabla \Sigma^{\RR}_{r})$ at general points, then inductively,
\[
\rank(\nabla \Sigma^{\RR}_{r-1}) =\rank(\nabla \Sigma^{\RR}_{r}) = \rank(\nabla \Sigma^{\RR}_{r+1}) = \cdots
\]
at general points, which implies that
\[
\dim (\Ima{\Sigma^{\RR}_{r-1}}) = \dim (\Ima{\Sigma^{\RR}_{r}}) = \cdots = n_Un_Vn_W.
\]
Hence if $r < r_g$, $\rank(\nabla \Sigma^{\RR}_{r-1}) < \rank(\nabla \Sigma^{\RR}_{r})$ at general points, implying that
\[
\dim (\Ima{\Sigma^{\RR}_{r-1}}) < \dim (\Ima{\Sigma^{\RR}_{r}}).
\]
On the other hand, since $(U_+ \times V_+ \times W_+)^{r}$ contains an open subset of $(U \times V \times W)^{r}$, by Lemma~\ref{thm:sard}, $\nabla \Sigma^{\RR_+}_r = \nabla \Sigma^{\RR}_r$ at a general point, $\Ima{\Sigma^{\RR_+}_r}$ contains an open subset of $\Ima{\Sigma^{\RR}_r}$, i.e.,
\begin{multline*}
\dim (D_{r-1}) = \dim (\Ima{\Sigma^{\RR_+}_{r-1}}) = \dim (\Ima{\Sigma^{\RR}_{r-1}}) \\
< \dim (\Ima{\Sigma^{\RR}_{r}}) = \dim (\Ima{\Sigma^{\RR_+}_{r}})= \dim (D_r).
\end{multline*}
Thus a general $A \in D_r$ has nonnegative rank $r$, and the real rank of $A$ is also $r$.
\end{proof}

We now relate real rank to complex rank (and later to nonnegative rank) via general relations between real algebraic varieties and their complexifications. For a field of characteristic zero  $\mathbb{K}$, we write $\mathbb{KP}^n$ for the projective space of dimension $n$ over $\mathbb{K}$. As we briefly mentioned after Definition~\ref{def:cone}, the affine cone of a projective variety $X \subseteq \mathbb{KP}^n$ is the affine variety
\[
\widehat{X} \coloneqq \{ x \in \mathbb{K}^{n+1} : \pi(x) \in X\} \cup \{0\} = \pi^{-1}(X) \cup \{0\},
\]
where $\pi :  \mathbb{K}^{n+1}  \to\mathbb{KP}^n$ is the natural projection that takes a point $x  \in   \mathbb{K}^{n+1}$ to the equivalence class $\pi(x) = \{ \lambda x \in  \mathbb{K}^{n+1} : \lambda \in \mathbb{K}^\times \} \in \mathbb{KP}^n$.
\begin{definition}\label{def:join}
Let $X , Y \subseteq \mathbb{KP}^n$ be projective varieties.
Let $\varphi \colon \widehat{X} \times \widehat{Y} \to \mathbb{K}^{n+1}$, $(x, y) \mapsto x + y$. The \emph{join} of $X$ and $Y$ is the projective variety $ J(X, Y) \subseteq \mathbb{KP}^n$ whose affine cone is the Zariski closure of the image $\varphi(\widehat{X} \times \widehat{Y}) \subseteq \mathbb{K}^n$. The $k$th \emph{secant variety} of $X$ is the projective variety defined by
\[
\sigma^{\mathbb{K}}_k(X) \coloneqq 
\begin{cases}
 J(X, X) & \text{if} \; k = 2,\\
J\bigl(X, \sigma^{\mathbb{K}}_{k-1}(X)\bigr) & \text{if}\; k > 2.
\end{cases}
\]
\end{definition}
 We define
\begin{multline*}
\Var(\mathbb{RP}^n) \coloneqq \{ X \subseteq \mathbb{RP}^n : X \; \text{a real projective variety that is}\\
\text{(i) irreducible, (ii) nondegenerate, (iii) Zariski dense in}\; X(\CC)\}
\end{multline*}
Let $I(X) \subseteq \RR[x_0, \dots, x_n]$ be the homogeneous ideal of $X$ and $r_g(X)$ be the complex generic $X$-rank.  Standard elimination theory (see \cite[Section~2.1]{SidmanVermeire2011} and \cite[Section~2.2]{BlekTei14:ma}) yields the following relation between a real secant variety and its complexification.

\begin{lemma}\label{lem:rcgen}
Let $X \in \Var(\mathbb{RP}^n)$ and $r < r_g(X)$. Then there exists a set of homogeneous generators $f_1, \dots, f_m$ of the ideal $I\bigl(\sigma^{\RR}_r(X)\bigr)$ that also generates the ideal $I\bigl(\sigma^{\CC}_r(X({\CC}))\bigr)$. In particular, $\sigma^{\CC}_r(X({\CC}))$ is the complexification of $\sigma^{\RR}_r(X)$.
\end{lemma}

It is also not difficult to see the following relation between smooth points on a real secant variety and general points on its complexification.
\begin{lemma}\label{lem:realsmoothpt}
Let $X \in \Var(\mathbb{RP}^n)$ and $r < r_g(X)$. Then  $\sigma^{\RR}_r(X) \in \Var(\mathbb{RP}^n)$. 
\end{lemma}

\begin{proof}
 It suffices to show that at least one point in $\sigma^{\RR}_r(X)$ is a smooth point in $\sigma^{\CC}_r(X({\CC}))$.  Suppose not. Then $\sigma^{\RR}_r(X)$ is in the singular locus of $\sigma^{\CC}_r(X({\CC}))$. Let $k =\dim \sigma^{\CC}_r(X({\CC}))$. Then $\sigma^{\RR}_r(X)$ satisfies the equations given by the vanishing of the $(n -  k) \times (n - k)$ minors of
\[
\begin{bmatrix}
\partial f_1/\partial x_0 & \cdots & \partial f_1/\partial x_n \\
\vdots & \ddots & \vdots \\
\partial f_m/\partial x_0 & \cdots & \partial f_m/\partial x_n \end{bmatrix},
\]
which are defined over $\RR$. On the other hand, these minors are not all in $I\bigl(\sigma_r^{\RR}(X)\bigr)$ as $\sigma_r^{\RR}(X)$  itself  has at least one real smooth point --- a contradiction. Hence at least one point in $\sigma^{\RR}_r(X)$ is a smooth point of $\sigma^{\CC}_r(X({\CC}))$.
\end{proof}

By \cite[Corollary~1.8]{Adlandsvik87:ms}, $\sigma^{\CC}_{r-1}(X({\CC}))$ is in the singular locus of $\sigma^{\CC}_r (X({\CC}))$. Applying this to $X = \Seg(\PP U \times \PP V \times \PP W)$, the \emph{Segre variety} of rank-one tensors, we obtain the following from Lemma~\ref{lem:realsmoothpt}.

\begin{lemma}\label{lem:realcomp}
Let $r < r_g$. Then a general real tensor $A$ of real rank $r$ has complex rank $r$.
\end{lemma}
\begin{theorem}\label{thm:riden1}
 Let $r < r_g$. Then a general $A \in D_r$ has both real rank and complex rank equal to $r$. If $U \otimes V \otimes W$ is $r$-identifiable, then $A$ has a unique nonnegative rank-$r$ decomposition. 
\end{theorem}

\begin{proof}
%If $U \otimes V \otimes W$ is $r$-identifiable, then $U \otimes V \otimes W$ is not $r$-defective. 
The claims about ranks are just Lemmas~\ref{lem:nonnegreal} and \ref{lem:realcomp}. Since $D_r$ contains an open subset of $\Ima{\Sigma^{\RR}_r}$, a general point in $D_r$ has a unique rank-$r$  decomposition.
\end{proof}

There has been a significant amount of work on both defectivity \cite{Strassen83:laa, Lickteig85:laa, AboOP09:tams} and identifiability \label{StegemanLine} \cite{Krus77:laa,  Stegeman09:laa, ChiantiniOtt12:simax, DomanovDeLa131:simax, DomanovDeLa132:simax, BoChOtt14:ampa, ChiaOV14:simax, DomanovDeLa14:arxiv}. While these focus mainly on complex tensors, some of these methods can be also adapted to real tensors. Two notable examples are  \cite[Theorem~1.1]{ChiantiniOtt12:simax} and \cite[Proposition~1.6]{DomanovDeLa14:arxiv}, stated below for real tensors.

\begin{theorem}[Chiantini--Ottaviani]\label{thm:riden}
Let $U, V$, and $W$  be real vector spaces with dimensions $\dim U \le \dim V \le \dim W$. Let $\alpha, \beta$ be minimum integers such that $2^{\alpha} \le \dim U$ and $2^{\beta} \le \dim V$. Then $U \otimes V \otimes W$ is $r$-identifiable if $r \le 2^{\alpha+\beta-2}$.
\end{theorem}

\begin{theorem}[Domanov--De~Lathauwer]\label{thm:DomanDeL}
Let  $U, V$, and $W$ be real vector spaces with dimensions $\dim U =m$, $\dim V =n$, and $\dim W = p$. If
\[
2 \le m \le n \le p \le r \qquad \text{and} \qquad 2r \le  m+n+2p-2-\sqrt{(m-n)^2+4p},
\]
then $U \otimes V \otimes W$ is $r$-identifiable.
\end{theorem}

Applying Theorem~\ref{thm:riden} to  Theorem~\ref{thm:riden1}, we obtain explicit examples.
\begin{corollary}\label{cor:unique1}
Let $n \geq 4$ and $r \le \floor{n^2/16}$. A general $A \in \RR_+^{n \times n \times n}$ with $\rank_+(A) = r$ has complex rank $r$ (and therefore real rank $r$) and a unique nonnegative rank-$r$  decomposition.
\end{corollary}

 In fact we may also derive identifiability results for real tensors from the identifiability results for complex tensors.

\begin{lemma}\label{lem:rcxrkiden}
Let $X \in \Var(\mathbb{RP}^n)$ and $r < r_g(X)$. If a general point in $\sigma^{\CC}_r(X({\CC}))$ has a unique rank-$r$ decomposition, then a general point in $\sigma^{\RR}_r(X)$ has a unique complex rank-$r$ decomposition.
\end{lemma}

\begin{proof}
Suppose not, then there is some nonempty Euclidean open subset $\mathcal{U}$ of $\sigma^{\RR}_r(X)$ such that any point in $\mathcal{U}$ does not have a unique complex rank-$r$ decomposition. By assumption, the set of points in $\sigma^{\CC}_r(X({\CC}))$ that do not have unique rank-$r$ decompositions is contained in a subvariety  $Y \subseteq \sigma^{\CC}_r(X({\CC}))$. Then $\mathcal{U} \subset Y$, and so the Zariski closure of $\mathcal{U}$, i.e., $\sigma^{\RR}_r(X)$, is contained in $Y$. But by Lemma~\ref{lem:realsmoothpt}, $\sigma^{\RR}_r(X)$ is Zariski dense in $\sigma^{\CC}_r(X({\CC}))$, a contradiction.
\end{proof}

Lemma~\ref{lem:rcxrkiden} does not guarantee that a general point in $\sigma^{\RR}_r(X)$ has a unique \emph{real} rank-$r$ decomposition as there may be a Euclidean open subset in $\sigma^{\RR}_r(X)$ where every point has real rank greater than $r$. We now apply Lemma~\ref{lem:rcxrkiden} to the case $X = \Seg(\PP U \times \PP V \times \PP W)$.

\begin{theorem}\label{thm:cmplxidrliden}
Let $U, V$, and $W$  be real vector spaces  and let $r < r_g$. If $U_{\CC} \otimes V_{\CC} \otimes W_{\CC}$ is $r$-identifiable, then $U \otimes V \otimes W$ is $r$-identifiable.
\end{theorem}

\begin{proof}
If $U_{\CC} \otimes V_{\CC} \otimes W_{\CC}$ is $r$-identifiable, then a general point in $\sigma^{\CC}_r (\Seg (\PP U_{\CC} \times \PP V_{\CC} \times \PP W_{\CC}))$ has a unique complex rank-$r$ decomposition. By Lemma~\ref{lem:rcxrkiden}, a general point in $\sigma^{\RR}_r (\Seg (\PP U \times \PP V \times \PP W))$ has a unique complex rank-$r$ decomposition. Since $\Ima{\Sigma^{\RR}_r}$ contains a Euclidean open subset of $\sigma^{\RR}_r (\Seg (\PP U \times \PP V \times \PP W))$, a general point $A\in \Ima{\Sigma^{\RR}_r}$ has real rank $r$ and a unique complex rank-$r$ decomposition. By Lemma~\ref{lem:realcomp}, $A$ has complex rank $r$; and so the unique complex rank-$r$ decomposition of $A$ is in fact its unique real rank-$r$ decomposition. Therefore $U \otimes V \otimes W$ is $r$-identifiable.
\end{proof}
A consequence of Theorem~\ref{thm:cmplxidrliden} is the following corollary of \cite[Theorem~1.1]{ChiaOV14:simax}.
\begin{corollary}
Let $n_1 \ge \dots \ge n_d$ and
\[
r_0 = \biggl\lceil\frac{\prod_{i=1}^d n_i}{1 + \sum_{i=1}^d (n_i-1)}\biggr\rceil.
\]
Then $\RR^{n_1 \times \cdots \times n_d}$ is $r$-identifiable for $r < r_0$ if $\prod_{i=1}^d n_i \le 15000$ and $(n_1, \dots, n_d, r)$ is not one of the following cases:
\begin{center}
\begin{tabular}{|c|c|}
  \hline
  $(n_1, \dots, n_d)$ & $r$ \\
  \hline\hline
  $(4, 4, 3)$ & $5$ \\
  \hline
  $(4, 4, 4)$ & $6$ \\
  \hline
  $(6, 6, 3)$ & $8$ \\
  \hline
  $(n, n, 2, 2)$ & $2n - 1$ \\
  \hline
  $(2, 2, 2, 2, 2)$ & $5$ \\
  \hline
  $n_1 > \prod_{i=2}^d n_i - \sum_{i=2}^d (n_i - 1)$ & $r \ge \prod_{i=2}^d n_i - \sum_{i=2}^d (n_i - 1)$ \\
  \hline
\end{tabular}
\end{center}
\end{corollary}

By Lemma~\ref{lem:realsmoothpt}, we may  also  apply the algorithm proposed in \cite{ChiaOV14:simax} for complex tensors to  directly  test if a general real tensor of real rank-$r$ or a general nonnegative tensor of nonnegative rank-$r$ has a unique complex rank-$r$ decomposition. The sufficient condition to ensure the smoothness of a specific complex tensor in \cite[Lemma~5.1]{ChiaOV14:simax} may also be adapted to real tensors.

This discussion would not be complete without examples of non-identifiability cases. As most of the non-identifiability cases in the literature are for the complex case, we provide a result that allows us to translate them to the real case.
\begin{lemma}\label{lem:complxrealrdef}
Let $V_1, \dots, V_d$ be real vector spaces of dimensions $n_1, \dots, n_d$ respectively. Let $U_1, \dots, U_d$ be their complexifications, i.e., $U_i = V_i \otimes_{\RR} \CC$, $i = 1, \dots, d$. If $U_1 \otimes \cdots \otimes U_d$ is $r$-defective and $r < r_g$, then $V_1 \otimes \cdots \otimes V_d$ is also $r$-defective.
\end{lemma}

\begin{proof}
Let $A = \sum_{i=1}^r v^{(1)}_i \otimes \cdots \otimes v^{(d)}_i \in V_1 \otimes \cdots \otimes V_d$ be a general real rank-$r$ tensor. Let
\[
X \coloneqq \Seg(\PP V_1 \times \cdots \times \PP V_d)\qquad \text{and}\qquad X({\CC}) \coloneqq \Seg(\PP U_1 \times \cdots \times \PP U_d).
\]
By our semialgebraic Terracini's lemma, i.e., Lemma~\ref{lem:semiterracinilem},
\[
\mathsf{T}_{A} \widehat{\sigma}_r^{\RR} (X) = \spa_\RR \{ V_1 \otimes v^{(2)}_1 \otimes \cdots \otimes v^{(d)}_1, \dots, v^{(1)}_r \otimes \cdots \otimes v^{(d-1)}_{r} \otimes V_d \}.
\]
By Lemma~\ref{lem:realsmoothpt}, $A$ is a smooth point of $\sigma_r^{\CC} (X({\CC}))$, and thus by the usual complex Terracini's lemma,
\[
\mathsf{T}_{A} \widehat{\sigma}_r^{\CC} (X({\CC})) = \spa_\CC \{ U_1 \otimes v^{(2)}_1 \otimes \cdots \otimes v^{(d)}_1, \dots, v^{(1)}_r \otimes \cdots \otimes v^{(d-1)}_{r} \otimes U_d \}.
\]
By assumption,
\[
\dim_{\CC} \mathsf{T}_{A} \widehat{\sigma}_r^{\CC} (X({\CC})) < r(n_1 + \cdots + n_d - d + 1),
\]
i.e., there exist $u^{(k)}_1, \dots, u^{(k)}_r \in U_i$ with $[u^{(k)}_i] \neq [v^{(k)}_i] \in \PP U_i$ for $k = 1, \dots, d$, $i = 1, \dots, r$, and
\[
u^{(1)}_1  \otimes v^{(2)}_1 \otimes \cdots \otimes v^{(d)}_1 + \cdots + v^{(1)}_r \otimes \cdots  \otimes v^{(d-1)}_r\otimes u^{(d)}_r = 0.
\]
By taking the real part or the imaginary part of each $u^{(k)}_i$, we have $\dim_{\RR} \mathsf{T}_{A} \widehat{\sigma}_r^{\RR} (X) < r(n_1 + \cdots + n_d - d + 1)$, i.e., $V_1 \otimes \cdots \otimes V_d$ is $r$-defective.
\end{proof}
Using the corresponding results for complex tensors in \cite{AboOP09:tams, BoChOtt14:ampa} and Lemma~\ref{lem:complxrealrdef}, we deduce the following nonuniqueness result for real tensors.

\begin{theorem}\label{thm:COV}
\begin{enumerate}[\upshape (i)]
\item $\mathbb{R}^{4 \times 4 \times 3}$ is $5$-defective. So a general $4 \times 4 \times 3$ real tensor of real rank $5$ does not have a unique rank-$5$ decomposition over $\mathbb{R}$.

\item For any $n \ge 2$, $\mathbb{R}^{n \times  n \times 2 \times 2}$ is $(2n - 1)$-defective. So a general $n \times n \times n \times 2$ real tensor of real rank $2n-1$ does not have a unique rank-$(2n-1)$ decomposition  over $\mathbb{R}$.

\item For $n_1 \ge \cdots \ge n_d \ge 2$, $\RR^{n_1 \times \cdots \times n_d}$ is $r$-defective if
\[
n_1 > \prod\nolimits_{i=2}^d n_i - \sum\nolimits_{i=2}^d (n_i - 1) \quad \text{and}\quad r \ge \prod\nolimits_{i=2}^d n_i - \sum\nolimits_{i=2}^d (n_i - 1).
\]
So a general $(n_1 \times \cdots \times n_d)$-real tensor of real rank $r < r_g$ does not have a unique rank-$r$ decomposition  over $\mathbb{R}$.
\end{enumerate}
\end{theorem}
A complex analogue of Theorem~\ref{thm:COV} may be found in \cite[Theorem~1.1]{ChiaOV14:simax}.

We may also apply the techniques in this section to obtain analogous results for real symmetric tensors. We will denote the set of real or complex symmetric $d$-tensors by $\mathsf{S}^d(\RR^{n})$ or  $\mathsf{S}^d(\CC^{n})$ respectively. We say $\mathsf{S}^d(\CC^{n})$ is $r$-identifiable if a general symmetric rank-$r$ tensor in $\mathsf{S}^d(\CC^{n})$ has a unique symmetric rank decomposition (also known as \emph{Waring decomposition}). Applying Lemma~\ref{lem:rcxrkiden} to $X = \nu_d (\RR \PP^n)$, the \emph{Veronese variety} of symmetric rank-one symmetric tensors, we  deduce the following.
\begin{theorem}\label{thm:symmcmplxidrliden}
Let $r < r_g(\nu_d (\RR \PP^n))$. If $\mathsf{S}^d(\CC^{n+1})$ is $r$-identifiable, then $\mathsf{S}^d(\RR^{n+1})$ is $r$-identifiable.
\end{theorem}

When $r < r_g(\nu_d (\RR \PP^n))$, the $r$-identifiability of $\mathsf{S}^d(\CC^{n+1})$ has been completely determined for all values of $r, d, n$ \cite[Theorem~1.1]{ChiantiniOttavianiVanniewenhoven15:tams}; this together with  Lemma~\ref{lem:rcxrkiden} gives us the following.
\begin{corollary}
$\mathsf{S}^d(\RR^{n+1})$ is $r$-identifiable when
\[
r < \biggl\lceil\frac{\binom{n+d}{d}}{n+1} \biggr\rceil
\]
and if $(d, n, r) \notin \{(6,2,9), (4,3,8), (3,5,9)\}$.
\end{corollary}
\begin{proof}
This follows from \cite{ChiatiniCiliberto06:jlms}, \cite[Theorem~1.1]{Ballico05:cejm}, \cite[Theorem~4.1]{Mella06:tams},  and \cite[Theorem~1.1]{ChiantiniOttavianiVanniewenhoven15:tams}.
\end{proof}

\section{Typical and maximum nonnegative ranks}\label{sec:typmax}

In this section, we investigate typical, maximum, and maximum nonnegative typical ranks, as defined in Definition~\ref{def:typical}. The following rephrases \cite[Proposition~6.2]{LimC09:jchemo} in the context of this article and may be viewed as a generalization of \cite[Theorem~3.1]{BocCarlRapa11:siamjmaa}.

\begin{proposition}\label{prop:lowcon}
Let $A \in U_+ \otimes V_+ \otimes W_+$ with $\rank_+(A) = r$. Then there is an open ball $B(A, \varepsilon) \subseteq U \otimes V \otimes W$ such that
\[
\rank_+(A') \ge r
\]
for all $A' \in B(A, \varepsilon) \cap U_+ \otimes V_+ \otimes W_+$.
\end{proposition}

It follows immediately that the maximum nonnegative typical rank and the maximum nonnegative rank always coincide.
\begin{lemma}\label{lem:maxnonnegtyprk}
If $r$ is the maximum nonnegative rank of $U_+ \otimes V_+ \otimes W_+$, then $r$ is the maximum nonnegative typical rank.
\end{lemma}

What about the \emph{minimum} nonnegative typical rank then? It turns out that it is always equal to the (complex) generic rank.
\begin{lemma}\label{lem:mintyprk}
The minimum nonnegative typical rank of $U_+ \otimes V_+ \otimes W_+$ is the complex generic rank $r_g$ of $U_\mathbb{C} \otimes V_\mathbb{C} \otimes W_\mathbb{C}$.
\end{lemma}

\begin{proof}
Since $(U_+ \times V_+ \times W_+)^{r}$ contains an open subset of $(U \times V \times W)^{r}$, by Lemma~\ref{thm:sard}, $\rank(\nabla \Sigma^{\RR_+}_r) = \rank(\nabla \Sigma^{\RR}_r)$ at general points. Hence $\dim \Ima(\Sigma^{\RR_+}_r) = \dim \Ima(\Sigma^{\RR}_r)$, which implies that $r_g$ is the minimum nonnegative typical rank.
\end{proof}

\label{page222eg} We will illustrate these with a $2 \times 2 \times 2$ example. In this case, the complex generic rank of $\mathbb{C}^{2 \times 2 \times 2} $ is $2$ and the real typical ranks of $\mathbb{R}^{2 \times 2 \times 2}$ are $2$ and $3$ \cite{DesiL08:simax}. By Lemmas~\ref{le:interval}, \ref{lem:maxnonnegtyprk}, and \ref{lem:mintyprk}, to completely determine the  nonnegative typical ranks of $\mathbb{R}_+^{2 \times 2 \times 2} $, it remains to find the maximum nonnegative rank. We will construct a nonnegative tensor with maximum nonnegative rank explicitly. Consider the tensor
\begin{equation}\label{eq:tdecomp}
A = e_1 \otimes e_1 \otimes e_1 + e_2 \otimes e_2 \otimes e_1 + e_1 \otimes e_2 \otimes e_2 + e_2 \otimes e_1 \otimes e_2
\end{equation}
that we saw earlier in \eqref{eq:tdecomp0}.
$A$ may be represented by a nonnegative hypermatrix
\[
A = \left[
\begin{array}{cc|cc}
1 & 0 & 0 & 1 \\
0 & 1 & 1 & 0 \\
\end{array}
\right] \in \RR^{2 \times 2 \times 2}_+.
\]
Now let $A = \sum_{k=1}^r x_k \otimes y_k \otimes z_k$ be a nonnegative rank-$r$  decomposition. Then we must be able to write $A = \sum_{k=1}^{r'} X_k \otimes z_k$ where each $X_k$ is a nonnegative matrix. \label{2by2by2case} Observe that $z_k$ cannot be of the form $\alpha e_1 + \beta e_2$ where $\alpha, \beta > 0$. Otherwise by the nonnegativity of each $z_k$ and $X_k$, there is some $i, j \in \{1, 2\}$ such that the $(i, j, 1)$th coordinate and the $(i, j, 2)$th coordinate of $A$ are both positive, which contradicts the construction of $A$. Hence we must have $z_k = e_1$ or $e_2$ for all $k =1,\dots,r'$. So without loss of generality we may assume that $z_1 = e_1$ and $z_2 = e_2$. Then $X_1 = e_1 \otimes e_1+  e_2 \otimes e_2$ and $X_2 = e_1 \otimes e_2 + e_2 \otimes e_1$. By the uniqueness of the nonnegative decompositions of $X_1$ and $X_2$, the nonnegative rank-$r$  decomposition of $A$  in \eqref{eq:tdecomp} is unique. Hence $\rank_+(A) = 4$. Since any $T\in \RR_+^{2 \times  2 \times 2}$ has the form $T = Y_1 \otimes e_1 + Y_2 \otimes e_2$ where $Y_1, Y_2$ are nonnegative matrices, and the nonnegative rank of a nonnegative $2 \times 2$ matrix is at most $2$, we may conclude that the nonnegative rank of $T$ is at most $4$. Thus the nonnegative typical ranks of $\mathbb{R}_+^{2 \times 2 \times 2} $ are $2$, $3$, and  $4$.

Both the real and complex ranks of $A$ are $2$ \cite{DesiL08:simax}. In fact for any $A'$ in a sufficiently small open ball $B(A, \varepsilon)$, both the real and complex ranks of $A'$ are also $2$. If in addition, $A' \in B(A, \varepsilon) \cap(\mathbb{R}_+^{2 \times 2 \times 2} )$, then the nonnegative rank of $A'$ is $4$. This example can be generalized as follows.

\begin{lemma}\label{lem:permma}
Let $P_1, \dots, P_n \in \mathbb{R}^{n \times n}_+ \cong  \RR^n_+ \otimes \RR^n_+$ be $n$ permutation matrices such that for each $(i, j) \in \{1,\dots,n\} \times \{1,\dots, n\}$, there is one and only one $P_k$ whose $(i, j)$th entry is one. Let $e_1, \dots, e_n \in \mathbb{R}^n_+ $ be the standard basis of $\RR^n$. Define
\[
A = P_1 \otimes e_1 + \cdots + P_n \otimes e_n \in \mathbb{R}_+^{n \times n \times n} .
\]
Then $\rank_+(A) = n^2$ and $A$ has a unique nonnegative rank-$n^2$  decomposition.
\end{lemma}

\begin{proof}
It suffices to show that $A$ has a unique nonnegative rank-$n^2$  decomposition. Suppose
\[
A = \sum_{i=1}^{n^2} \Biggl[\sum_{j=1}^n \alpha^j_i e_j\Biggr] \otimes \Biggl[\sum_{j=1}^n \beta^j_i e_j\Biggr] \otimes \Biggl[\sum_{j=1}^n \gamma^j_i e_j\Biggr]
\]
for nonnegative $\alpha^j_i, \beta^j_i, \gamma^j_i$. Without loss of generality, we may assume $\alpha^1_1, \beta^1_1, \gamma^1_1 \neq 0$. Since there is only one $P_k$ whose $(1,1)$th entry is nonzero, this $P_k$ must be $P_1$ and $\gamma^j_1 = 0$ for all $j > 1$. Repeating this procedure we may show that when we regard $A$ as a  nonnegative matrix in $\mathbb{R}^{n^2 \times n}_+ \cong \mathbb{R}_+^{n \times n} \otimes \mathbb{R}^n_+$, it has a unique nonnegative matrix factorization given by $A = P_1 \otimes e_1 + \cdots + P_n \otimes e_n$. Since each $P_k$ has a unique nonnegative matrix factorization \cite{LaurbergCPHJ08:cin}, $A$ has a unique nonnegative rank-$n^2$  decomposition.
\end{proof}

A $d$-tensor in $V_1\otimes\dots\otimes V_d$ is said to be \emph{cubical} if $\dim V_1 = \dots= \dim V_d$.  By \cite[Theorem~4.4]{Lickteig85:laa}, \cite[Theorem~4.6]{Strassen83:laa}, Lemmas~\ref{le:interval}, \ref{lem:mintyprk}, \ref{lem:maxnonnegtyprk}, and \ref{lem:permma}, we completely determine the nonnegative typical ranks of cubical nonnegative tensors.
\begin{proposition}\label{prop:typical1}
For $n = 2$, the nonnegative typical ranks of $\mathbb{R}_+^{2 \times 2 \times 2} $ are given by all integers $m$ where
\[
2\le m \le 4.
\]
For $n = 3$, the nonnegative typical ranks of $\mathbb{R}_+^{3 \times 3 \times 3}$ are given by all integers $m$ where
\[
5\le m \le 9.
\]
For $n \ge 4$,
the nonnegative typical ranks of $\mathbb{R}_+^{n \times n \times n}$ are given by all integers $m$ where
\[
\left\lceil\frac{n^3}{3n-2}\right\rceil \le m \le n^2.
\]
\end{proposition}

For nonnegative tensors that are not cubical, we may determine the maximum nonnegative typical ranks but since the complex generic ranks for $3$-tensors are still not known in some instances, we do not have a complete list of nonnegative typical ranks.
\begin{proposition}\label{prop:typical2}
Write $\operatorname{maxrank}_+(m,n,p)$ for the maximum nonnegative typical rank of $\mathbb{R}_+^{m \times n \times p}$ and suppose without loss of generality that $m \ge n \ge p$. Then
\[
\operatorname{maxrank}_+(m,n,p) =
\begin{cases}
np & \text{if}\; m = n \ge p,\\
n^2 &\text{if}\; m \ge n = p,\\
np & \text{if}\; m > n > p.
\end{cases}
\]
\end{proposition}

\begin{proof}
The required arguments are as in the proof of Lemma~\ref{lem:permma} but `padded with the appropriate number of zeros,' i.e., applied to matrices of the form
\[
\begin{bmatrix}
P_k\\
0 
\end{bmatrix} \quad \text{or} \quad
\begin{bmatrix}
P_k & 0 
\end{bmatrix}
\]
where $P_k$ is a permutation matrix.
\end{proof}

\section{General uniqueness of decompositions of approximations}\label{sec:generic}

In our previous work \cite{QComonLim14:arxiv}, we established that a general nonnegative tensor has a unique best nonnegative rank-$r$ approximation. Here we investigate whether this best nonnegative rank-$r$ approximation has a unique nonnegative rank-$r$ decomposition.

Let $U,V,W$ be real vector spaces of dimensions $n_U,n_V,n_W$ respectively. We will assume a choice of basis on these vector spaces,  so that $U \cong \mathbb{R}^{n_U}$, $V \cong \mathbb{R}^{n_V}$, and $W \cong \mathbb{R}^{n_W}$. For a vector $u_i \in U$, we let $u_{i, j}$ denote the $j$th coordinate of $u_i$. Likewise for $V$ and $W$. For any smooth curve $\gamma(t)$, $t \in [0, 1]$, the right derivative at $0$ is denoted by
\[
\gamma'(0) \coloneqq
\lim_{t \to 0^+} \frac{\gamma(t) - \gamma(0)}{t - 0}.
\]
Recall the map $\Sigma^{\RR_+}_r \colon (U_+ \times V_+ \times W_+ )^r \to U_+ \otimes V_+ \otimes W_+ $ defined in \eqref{eq:sigma} and \eqref{eq:sigma2}. The \emph{pushforward} of $\Sigma^{\RR_+}_r$ at $\gamma'(0)$ is denoted
\[
\Sigma^{\RR_+}_{r*}\bigl(\gamma'(0)\bigr) \coloneqq
\lim_{t \to 0^+} \frac{\Sigma^{\RR_+}_r\bigl(\gamma(t)\bigr) - \Sigma^{\RR_+}_r\bigl(\gamma(0)\bigr)}{t - 0}.
\]

Let $S_r \subseteq U_+ \otimes V_+ \otimes W_+ $ denote the set of nonnegative tensors on which the distance function  $\dist (\cdot, D_r)$ is not smooth. Then $S_r$ contains the nonnegative tensors with non-unique best nonnegative rank-$r$ approximations and is a nowhere dense semialgebraic subset \cite{FrieS16:banach}. Let $\pi_r \colon U_+ \otimes V_+ \otimes W_+ \setminus S_r \to D_r$ be the map sending a nonnegative tensor to its unique best nonnegative rank-$r$ approximation. Since the distance function $\dist(\cdot, D_r)$ is semialgebraic \cite{Coste02:raag, FrieS16:banach}, the graph of $\pi_r$,
\[
G(\pi_r) = \{ (p, q) \in (U_+ \otimes V_+ \otimes W_+ \setminus S_r) \times D_r : \dist(p, D_r) = \norm{p - q} \},
\]
is also semialgebraic. By \label{pageprop6} Proposition~\ref{prop:smooth}, the subset of points in $U_+ \otimes V_+ \otimes W_+ \setminus S_r$ where $\pi_r$ is not smooth is contained in a hypersurface $H_r$. Henceforth we will focus on the restriction of $\pi_r$ (also denoted $\pi_r$ with a slight abuse of notation) to a subset of smooth points in $ U_+ \otimes V_+ \otimes W_+$,
\[
\pi_r \colon U_+ \otimes V_+ \otimes W_+ \setminus (S_r \cup H_r) \to D_r.
\]

In the following the \emph{support} of a vector $u \in U$ is defined to be
\[
\operatorname{supp}(u) \coloneqq \{ i \in \{1,\dots,n_U \} : u_i \ne 0\}.
\]
The next lemma is a slight rephrase of \cite[Lemma~13]{QComonLim14:arxiv}. We will use it to partition $D_r$ into a union of semialgebraic sets later.
\begin{lemma}\label{le: perp}
Let $p \in U_+ \otimes V_+ \otimes W_+ \setminus (S_r \cup H_r)$ where $\pi_r(p)$ has a nonnegative rank-$r$ decomposition
\begin{equation}\label{eq:phi}
\pi_r(p) = \sum_{i = 1}^r u_i \otimes v_i \otimes w_i.
\end{equation}
Then for any  $x_i \in U_+$, $i = 1, \dots, r$, we have
\begin{equation}\label{ineq:tan}
\left\langle p, x_i \otimes v_i \otimes w_i \right\rangle \le \left\langle \pi_r(p), x_i \otimes v_i \otimes w_i \right\rangle,
\end{equation}
where $\langle\cdot,\cdot\rangle$ denotes the Euclidean inner product.
With respect to the nonnegative vectors $u_1,\dots,u_r$ in \eqref{eq:phi}, define the subspaces
\begin{equation}\label{eq:tilde}
\widetilde{U}_i \coloneqq \{u \in U : \operatorname{supp}(u) \subseteq \operatorname{supp}(u_{i}) \}
\end{equation}
for $i =1,\dots,r$, and define $\widetilde{V}_i$ and $\widetilde{W}_i$ similarly. Then for $x_i \in \widetilde{U}_i$, $i=1,\dots,r$, we have
\begin{equation}\label{eq:nontan}
\langle p, x_i \otimes v_i \otimes w_i \rangle = \left\langle \pi_r(p), x_i \otimes v_i \otimes w_i \right\rangle.
\end{equation}
The analogous statement for $\widetilde{V}_i$ or $\widetilde{W}_i$  in place of $\widetilde{U}_i$ holds true as well.
\end{lemma}

%\begin{proof}
%Let $i \in \{1,\dots,r\}$ be fixed. Consider a curve through $X(0) = \pi_r(p)$ given by
%\[
%X(t) = (u_i+t x_i) \otimes v_i \otimes w_i + \sum_{k \ne i} u_k \otimes v_k \otimes w_k,
%\]
%for an arbitrary $x_i \in U_+$. Note that for $t \ge 0$, $\norm{p-X(t)}$ achieves a local minimum at $t=0$, and is therefore nondecreasing in $[0, \varepsilon)$ for some small $\varepsilon > 0$. Hence the right derivative of $\norm{p-X(t)}$ at $0$ is nonnegative, i.e.,
%\[
%\lim_{t \to 0^+}\frac{d}{dt} \norm{p-X(t)} \ge 0,
%\]
%and thus $\langle p, x_i \otimes v_i \otimes w_i \rangle \le \left\langle \pi_r(p), x_i \otimes v_i \otimes w_i \right\rangle$.
%
%Now let $x_i \in \widetilde{U}_i$. Then $X(t)$ is nonnegative for $t \in (-\varepsilon, \varepsilon)$ and the local minimality of $\norm{p-X(t)}$ at $t=0$ implies
%\[
%\frac{d}{dt}  \norm{p-X(t)} \Bigr|_{t=0} = 0,
%\]
%and thus $\langle p, x_i \otimes v_i \otimes w_i \rangle = \left\langle \pi_r(p), x_i \otimes v_i \otimes w_i \right\rangle$.
%\end{proof}

We first remind the reader of our abbreviated notation in \eqref{eq:abbre}. Let
\[
\mathsf{T}_{\pi_r(p)} (u_1, \dots, w_r) \coloneqq \operatorname{span}_\mathbb{R} \Bigl(\bigcup\nolimits_{i=1}^r \widetilde{U}_i \otimes v_i \otimes w_i \cup u_i \otimes \widetilde{V}_i \otimes w_i \cup u_i \otimes v_i \otimes \widetilde{W}_i \Bigr).
\]
By Lemma~\ref{lem:semiterracinilem}, this is the tangent space of $D_r$ at $\pi_r(p)$ when $\pi_r(p)$ is a smooth point of $D_r$. Then \eqref{eq:nontan} implies that\footnote{Our convention: $\langle S, u \rangle = \langle u, S \rangle = 0$ for $S \subseteq U$ means that every vector in $S$ is orthogonal to $u$; $\langle S, T \rangle = 0$ for $S,T \subseteq U$ means that any vector in $S$ is orthogonal to any vector in $T$.}
\begin{equation}\label{eq:tanperp}
\langle \mathsf{T}_{\pi_r(p)}(u_1, \dots, w_r), p - \pi_r(p) \rangle = 0,
\end{equation}
i.e., $ p - \pi_r(p)$ is orthogonal to the subspace $\mathsf{T}_{\pi_r(p)}(u_1, \dots, w_r)$.

%%%%%%%%
Let $\sigma_r$ denote the Euclidean closure of $\Ima{\Sigma^{\RR}_r}$. Then $D_r \subseteq \sigma_r$. By the Tarski--Seidenberg Theorem, $\sigma_r$ is semialgebraic. By \cite[Theorem~3.7]{FrieS16:banach}, a general $A \in U \otimes V \otimes W \setminus \sigma_r$ has a unique best approximation $\widetilde{\pi}_r(A)$ in $\sigma_r$. Note that for a nonnegative $A$, $\widetilde{\pi}_r(A) \in \sigma_r$ may be different from $\pi_r(A) \in D_r$.
%If $\widetilde{\pi}_r(A) \neq \pi_r(A)$, then $\rank_+{(\widetilde{\pi}_r(A))} > r$. Now we will take a closer look at what $\pi_r(A)$ looks like in this case.

In order to study best nonnegative rank approximations, i.e., the image of $\pi_r$, we first partition $D_r$ into a union of special semialgebraic subsets. For any index set $I_i \subseteq \{1,\dots,n_U\}$, let
\[
U_+(I_i) \coloneqq \{ u \in U_+ :  \operatorname{supp}(u) = I_i^c  \}
\]
and likewise for $V_+(J_i)$ and $W_+(K_i)$ with index sets $J_i \subseteq \{1,\dots,n_V\}$ and $K_i \subseteq \{1,\dots,n_W\}$. Here $I_i^c \coloneqq \{1, \dots, n_U\} \setminus I_i$ denotes set-theoretic complement. Given tuples of index sets
\[
I= (I_1, \dots, I_r), \quad J = (J_1, \dots, J_r), \quad K =(K_1, \dots, K_r)
\]
with $I_i \subseteq \{ 1, \dots, n_U \}$, $J_i \subseteq \{ 1, \dots, n_V \}$, $K_i \subseteq \{ 1, \dots, n_W \}$, $i =1,\dots,r$, we define  a \emph{cell} of $D_r$ corresponding to these index sets by
\begin{multline*}
D_r(I, J, K)  \coloneqq \biggl\{ A \in D_r : A = \sum\nolimits_{i=1}^r u_i \otimes v_i \otimes w_i, \\
u_i \in U_+(I_i), \; v_i \in V_+(J_i), \; w_i \in W_+(K_i), \; i=1,\dots,r \biggr\}.
\end{multline*}
The notion of a cell is important for our study of uniqueness because of the following easy observation.
\begin{lemma}\label{lem:distnonunique}
Let $A \in D_r$. If $A$ belongs to distinct cells, then the nonnegative $r$-term decomposition of $A$ is not unique.
\end{lemma}

Clearly, if $I_i = J_i = K_i = \varnothing$ for all $i = 1, \dots, r$, then $\dim D_r(I, J, K)  = \dim D_r$ and we call this the \emph{trivial cell}. The union of all nontrivial cells is called the \emph{boundary} of $D_r$, and denoted by $\partial D_r$.
\begin{lemma}\label{lem:boundary}
If $r < r_g$ and $U \otimes V \otimes W$ is not $r$-defective, then $\dim \partial D_r < \dim D_r$.
\end{lemma}
\begin{proof}
We first describe $\partial D_r$ explicitly. Let $\alpha \in \{1, \dots, n_U\}$ and $i \in \{1,\dots, r\}$. Let $\widetilde{U}_+(\alpha) = \{ u \in U_+ \colon \alpha \notin \operatorname{supp}(u) \}$. Define
\[
\partial D^{(i, \alpha)}_{r, U} \coloneqq \Sigma^{\RR_+}_r \bigl( (U_+ \times V_+ \times W_+)^{i-1} \times ( \widetilde{U}_+(\alpha)\times V_+ \times W_+) \times (U_+ \times V_+ \times W_+)^{r-i} \bigr).
\]
We write
\[
\partial D_{r, U} \coloneqq \bigcup_{i=1}^r \bigcup_{\alpha =1}^{n_U} \partial D^{(i, \alpha)}_{r, U}
\]
and likewise define $\partial D_{r, V}$ and $\partial D_{r, W}$.
The boundary is then the union of these three semialgebraic subsets,
\[
\partial D_r =  \partial D_{r, U}  \cup \partial D_{r, V} \cup \partial D_{r, W}.
\]
%\[
%\partial D_r = \bigcup_{\substack{ i \in \{1, \dots, r\} \\ \alpha \in \{1, \dots, n_U\} }} \partial D^{(i, \alpha)}_{r, U} \cup \bigcup_{\substack{ k \in \{1, \dots, r\} \\ \beta \in \{1, \dots, n_V\} }} \partial D^{(j, \beta)}_{r, V} \cup \bigcup_{\substack{ m \in \{1, \dots, r\} \\ \gamma \in \{1, \dots, n_W\} }} \partial D^{(k, \gamma)}_{r, W}.
%\]
From this description of $\partial D_r$, the required result is evident.
\end{proof}

We caution our reader that our notion of boundary of $D_r$ differs from both its topological boundary and its algebraic boundary  as defined in \cite{AllRhoSturmZw15:laa}.

%In particular, when $I_i = \{\alpha\}$ for some $\alpha \in \{1, \dots, n_U\}$, $I_j =\varnothing$ for all $j \ne i$, and $J_j = K_j = \varnothing$ for all $j$, the cell $D_r(I, J, K) $ is the image
%\[
%\Sigma^{\RR_+}_r \bigl( (U_+ \times V_+ \times W_+)^{i-1} \times ( U_+(\{\alpha\})\times V_+ \times W_+) \times (U_+ \times V_+ \times W_+)^{r-i} \bigr),
%\]
%and denoted by $\partial D^{(i, \alpha)}_{r, U}$. Let
%\[
%\partial D_r \coloneqq \bigcup_{\substack{ i \in \{1, \dots, r\} \\ \alpha \in \{1, \dots, n_U\} }} \partial D^{(i, \alpha)}_{r, U} \cup \bigcup_{\substack{ k \in \{1, \dots, r\} \\ \beta \in \{1, \dots, n_V\} }} \partial D^{(k, \beta)}_{r, V} \cup \bigcup_{\substack{ m \in \{1, \dots, r\} \\ \omega \in \{1, \dots, n_W\} }} \partial D^{(m, \omega)}_{r, W},
%\]
%and call $\partial D_r$ the \emph{boundary} of $D_r$. Note when $r < r_g$, $\dim \partial D_r < \dim D_r$. Each cell $D_r(I, J, K) $ is a semialgebraic subset of $\partial D_r$. For a positive $A$, we will see in Proposition~\ref{prop:mapopen} that if $\pi_r(A)$ is in some cell $D_r(I, J, K) $, then $A$ has an open neighborhood $\mathcal{V}$ such that $\pi_r(\mathcal{V}) \subseteq D_r(I, J, K) $. And we will study some concrete examples of cells $D_r(I, J, K) $ in Theorem~\ref{thm:uni23}.

Let $A \in U_+ \otimes V_+ \otimes W_+$ where $\pi_r(A)$ has a nonnegative rank-$r$ decomposition $\pi_r(A) = \sum_{i = 1}^r u_i \otimes v_i \otimes w_i$. If there is some $i \in \{ 1, \dots, r \}$ such that strict inequality holds in \eqref{ineq:tan}, i.e., there is some $x_i \in U_+$ with
\begin{align}
\left\langle A, x_i \otimes v_i \otimes w_i \right\rangle &< \left\langle \pi_r(A), x_i \otimes v_i \otimes w_i \right\rangle, \label{eq:tan1}\\
\intertext{then $\widetilde{\pi}_r(A) \neq \pi_r(A)$ and $\pi_r(A) \in \partial D_r$ by Lemma~\ref{le: perp}. Similarly, if}
  \left\langle A, u_i \otimes y_i \otimes w_i \right\rangle & < \left\langle \pi_r(A), u_i \otimes y_i \otimes w_i \right\rangle \label{eq:tan2} \\
\text{or}\quad  \left\langle A, u_i \otimes v_i \otimes z_i \right\rangle & <\left\langle \pi_r(A), u_i \otimes v_i \otimes z_i \right\rangle \label{eq:tan3}
\end{align}
for some $y_i \in V_+$ or $z_i \in W_+$, then $\widetilde{\pi}_r(A) \neq \pi_r(A)$ and $\pi_r(A) \in \partial D_{r}$. We define the following sets:
\begin{align}
\mathcal{L} &= \{\pi_r(A) \in \partial D_r : \pi_r(A) \; \text{satisfies \eqref{eq:tan1}, \eqref{eq:tan2}, or \eqref{eq:tan3}}\}, \label{eq:L}\\
\mathcal{N} &= \{A \in U_+ \otimes V_+ \otimes W_+ \setminus (S_r \cup H_r) : \pi_r(A)\in \mathcal{L} \}. \label{eq:N}
\end{align}
We will next show that every positive tensor (i.e., a tensor whose coordinates are positive) in $\mathcal{N}$ is an interior point.

\begin{proposition}\label{prop:boundary}
If $A \in \mathcal{N}$ is positive, then $A$ has an open neighborhood $\mathcal{V}$ such that $\mathcal{V} \subseteq \mathcal{N}$.
\end{proposition}

\begin{proof}
We first describe the structure of an open neighborhood $B(A, \eta)$ of a positive $A\in U_+ \otimes V_+ \otimes W_+$ and its image $\pi_r(B(A, \eta))$.
By \cite[Proposition~15]{QComonLim14:arxiv}, $\pi_r(A)$ always has nonnegative rank-$r$. 
Since $\pi_r$ is smooth, for any $\delta > 0$, there is some $\eta > 0$ such that $\pi_r(B(A, \eta)) \subseteq B(\pi_r(A), \delta) \cap D_r$.  Observe that $\bigl(\Sigma^{\RR_+}_r\bigr)^{-1}(B(\pi_r(A), \delta) \cap D_r)$ is a union of at most a countable number of products of open balls, say,
\[
\bigcup\nolimits_{j=1}^s \bigl( B(u^{(j)}_1, \delta^{(j)}_1) \cap U_+ \bigr) \times \cdots \times \bigl( B(w^{(j)}_r, \delta^{(j)}_r) \cap W_+ \bigr) \subseteq (U_+ \times V_+ \times W_+)^{r},
\]
where $s \in \mathbb{N} \cup \{ \infty \}$,  $u^{(j)}_i \in U_+$,  $v^{(j)}_i \in V_+$, $w^{(j)}_i \in W_+$, and $\delta^{(j)}_i > 0$ for $i =1,\dots,r$, and $j =1,\dots, s$. By dimension count, there exists some $j$ such that the image of
\[
\mathcal{U} \coloneqq (B(u^{(j)}_1, \delta^{(j)}_1) \cap U_+) \times \cdots \times (B(w^{(j)}_r, \delta^{(j)}_r) \cap W_+)
\]
under $\Sigma^{\RR_+}_r$ contains an open subset of $B(\pi_r(A), \delta) \cap D_r$. For notational convenience, we drop the superscript on  $u^{(j)}_i,v^{(j)}_i, w^{(j)}_i$ and write $u_i, v_i,w_i$ below.  By decreasing $\delta$ we may choose $\delta^{(j)}_1 = \cdots = \delta^{(j)}_r = \varepsilon$ for some $\varepsilon > 0$ small enough. Furthermore, we may assume that $\pi_r(A) = \sum_{i=1}^r u_i \otimes v_i \otimes w_i$ is a nonnegative rank-$r$ decomposition. So for any $p \in B(A, \eta)$, $\pi_r(p)$ has a nonnegative rank-$r$ decomposition $\pi_r(p) = \sum_{i=1}^r u_i(p) \otimes v_i(p) \otimes w_i(p)$ where 
\[
\norm{u_i - u_i(p)} \le \varepsilon, \quad \norm{v_i - v_i(p)}  \le \varepsilon, \quad \norm{w_i - w_i(p)}  \le \varepsilon,
\]
for $i=1,\dots,r$. Thus 
\begin{equation}\label{eq:suppine}
\operatorname{supp}(u_i) \subseteq \operatorname{supp}(u_i(p)), \quad \operatorname{supp}(v_i) \subseteq \operatorname{supp}(v_i(p)), \quad \operatorname{supp}(w_i) \subseteq \operatorname{supp}(w_i(p)),
\end{equation}
for $i=1,\dots,r$, and all $u_i(p)$, $v_i(p)$ and $w_i(p)$ depend continuously on $p$. The function defined by
\[
g(p) \coloneqq \langle p - \pi_r(p), x_i \otimes v_i(p) \otimes w_i(p) \rangle
\]
is therefore continuous on $B(A, \eta)$ for any fixed $x_i \in U_+$. If there is some $x_i \in U_+$ such that $\langle A - \pi_r(A), x_i \otimes v_i \otimes w_i \rangle < 0$, then by the continuity of $g$, there is an open neighborhood $\mathcal{V} \subseteq B(A, \eta)$ such $g(p) < 0$ for all $p \in \mathcal{V}$. Therefore $\mathcal{V} \subseteq \mathcal{N}$. 
\end{proof}

The following theorem is the main result of this section. It characterizes the relation between the image of $\pi_r$ and the cells of $D_r$. Its implication on nonnegative tensor decomposition and approximation will be given in Corollary~\ref{cor:uniqueness}.
\begin{theorem}\label{thm:mapopencell}
Let $\pi_r(A) \in D_r(I, J, K) $ for some cell $D_r(I, J, K)  \neq \{0\}$. Let $\mathcal{V}$ be an open neighborhood of $A$. Then $\pi_r(\mathcal{V})$ contains an open subset of $D_r(I, J, K) $.
\end{theorem}

\begin{proof}
We consider two cases: If $\pi_r(\mathcal{V}) $ is zero-dimensional, then we are led to a contradiction and so this case cannot occur. If $\pi_r(\mathcal{V}) $ is positive-dimensional, then we show that it must have full dimension in  $D_r(I, J, K) $  and therefore the required result follows.
\begin{case}
$\pi_r(\mathcal{V}) = \pi_r(A)$ is a point.
\end{case}

Let $\gamma(t)$ be a curve in $ \mathcal{V}$ with $\gamma(0) = A$. Then $\pi_r(\gamma(t)) = \pi_r(A)$ for any $t$. By \eqref{eq:tanperp} we have
\[
\langle \mathsf{T}_{\pi_r(A)}(u_1, \dots, w_r), \gamma(t) - \pi_r(A) \rangle = 0,\qquad
\langle \mathsf{T}_{\pi_r(A)}(u_1, \dots, w_r), A - \pi_r(A) \rangle = 0,
\]
implying that
\[
\langle \mathsf{T}_{\pi_r(A)}(u_1, \dots, w_r), \gamma(t) - A \rangle = 0.
\]
Since the curve $\gamma(t)$ is  arbitrary, we are led to the conclusion that
\[
\langle \mathsf{T}_{\pi_r(A)}(u_1, \dots, w_r), U \otimes V \otimes W \rangle = 0,
\]
contradicting the definition of $\mathsf{T}_{\pi_r(A)}(u_1, \dots, w_r)$.

\begin{case}
$\pi_r(\mathcal{V})$ is of positive dimension.
\end{case}

We will show that  $\dim \pi_r(\mathcal{V}) = \dim D_r(I, J, K) $. By \eqref{eq:suppine}, we may assume that $\pi_r(A)$ is a smooth point of $\pi_r(\mathcal{V})$  without loss of generality. By giving $\pi_r(\mathcal{V})$ a finer stratification, we may furthermore assume that $\pi_r(\mathcal{V})$ is a Nash manifold. Suppose that $\dim \pi_r(\mathcal{V}) < \dim D_r(I, J, K) $. Then by Theorem~\ref{thm:tubular} there is an open semialgebraic neighborhood $\mathcal{R}$ of $\pi_r(\mathcal{V})$ in $D_r(I, J, K) $ and a Nash retraction $f \colon \mathcal{R} \to \pi_r(\mathcal{V})$ such that
\[
\operatorname{dist}(p, \pi_r(\mathcal{V})) = \norm{p - f(p)}
\]
for any $p \in \mathcal{R}$. So there is a smooth curve $\gamma(t) \subseteq \mathcal{R}$ such that $\gamma(0) = \pi_r(A)$ and $f(\gamma(t)) = \pi_r(A)$. Let $A(t)  \coloneqq A - \pi_r(A) + \gamma(t)$ and $X(t) \coloneqq \pi_r(A(t)) \subseteq \pi_r(\mathcal{V})$. Note that
\[
\gamma(t), X(t) \subseteq D_r(I, J, K),\qquad A'(0), X'(0) \in \mathsf{T}_{\pi_r(A)}(u_1, \dots, w_r).
\]
By Lemma~\ref{le: perp},
\[
\lim_{t \to 0^+}\frac{d}{dt} \langle A(t) - X(t), A(t) - X(t) \rangle = 2 \langle A'(0) - X'(0), A - X(0) \rangle = 0.
\]
In fact, for any $s > 0$ small enough, we have
\[
\frac{d}{dt} \langle A(t) - X(t), A(t) - X(t) \rangle \Bigr|_{t=s} = 2 \langle A'(s) - X'(s), A(s) - X(s) \rangle = 0,
\]
implying that $\norm{A(t) - X(t)}$ is constant around $t=0$. On the other hand,
\[
\norm{A(t) - \gamma(t)} = \norm{A - \pi_r(A)}.
\]
So by the uniqueness of $\pi_r(A(t))$, $X(t) = \gamma(t)$, contradicting $\gamma(t) \subseteq \mathcal{R} \setminus \pi_r(\mathcal{V})$ for $t > 0$. Therefore we must have $\dim \pi_r(\mathcal{V}) = \dim D_r(I, J, K) $.
\end{proof}

%Here we can see that Theorem~\ref{thm:suniqueness} heavily depends on the geometric property of $D_r$. In fact it is not true if we replace $D_r$ by any semialgebraic subset $Z$. For example, when $Z$ has a cusp, then there may exist an open subset $U$ in the ambient space such that for each point $p$ of $U$, the best approximation of $p$ in $U$ is just the cusp.

%\begin{lemma}\label{lem:nonunicpd}
%Let $\Delta_r$ denote the set of nonnegative tensors with nonnegative rank $\le r$ and non-unique nonnegative rank-$r$ decompositions. Then $\Delta_r$ is semialgebraic.
%\end{lemma}
%
%\begin{proof}
%Define $\widetilde{\Delta}_r$ by {\footnotesize
%\begin{multline*}
%\widetilde{\Delta}_r = \bigg\{ \biggl( \Bigl(u_1, \dots, w_r, \sum_{i=1}^r u_i \otimes v_i \otimes w_i\Bigr), \Bigl(\widetilde{u}_1, \dots, \widetilde{w}_r, \sum_{i=1}^r \widetilde{u}_i \otimes \widetilde{v}_i \otimes \widetilde{w}_i\Bigr) \biggr) \in G(\Sigma_r) \times G(\Sigma_r) : \\
% \sum_{i=1}^r u_i \otimes v_i \otimes w_i = \sum_{i=1}^r \widetilde{u}_i \otimes \widetilde{v}_i \otimes \widetilde{w}_i, \\
%\sum_{i=1}^r \left( \lVert u_i - \widetilde{u}_{\sigma(i)}\rVert^2 + \lVert v_i - \widetilde{v}_{\sigma(i)}\rVert^2 + \lVert w_i - \widetilde{w}_{\sigma(i)}\rVert^2 \right) > 0, \; \text{for every} \; \sigma \in \mathfrak{S}_r \bigg\},
%\end{multline*}}
%then $\widetilde{\Delta}_r$ is semialgebraic. Since $\Delta_r$ is the projection of $\widetilde{\Delta}_r$ to $D_r$, $\Delta_r$ is semialgebraic.
%\end{proof}
\begin{corollary}\label{cor:uniqueness}
Let $r < r_g$, $U \otimes V \otimes W$ be $r$-identifiable, and $A \in U_+ \otimes V_+ \otimes W_+$ be general. If the unique best nonnegative rank-$r$ approximation $\pi_r(A)$ of $A$ is not in the boundary $\partial D_r$, then $\pi_r(A)$ has a unique nonnegative rank-$r$ decomposition.
\end{corollary}
\begin{proof}
Since $r < r_g$ and $U \otimes V \otimes W$ is not $r$-defective, by Lemma~\ref{lem:boundary},
\[
\dim \partial D_r < \dim D_r < \dim U \otimes V \otimes W.
\]
For any smooth point $q \in D_r$, there is an open neighborhood $\mathcal{Q} \subseteq D_r$ of $q$ such that any point in $\mathcal{Q}$ is also smooth. By Theorem~\ref{thm:tubular}, there is an open semialgebraic neighborhood $\mathcal{R}$ of $\mathcal{Q}$ in $U_+ \otimes V_+ \otimes W_+$ and a Nash retraction $f \colon \mathcal{R} \to \mathcal{Q}$ such that
$\operatorname{dist}(p, \mathcal{Q}) = \norm{p - f(p)}$
for every $p \in \mathcal{R}$. By shrinking $\mathcal{R}$ if necessary, we may assume that
\[
\norm{p - f(p)} = \operatorname{dist}(p, \mathcal{Q}) = \operatorname{dist}(p, D_r)
\]
for every $p \in \mathcal{R}$, i.e., $\pi_r(p) = f(p)$. Thus every smooth point of $D_r$ is contained in $\Ima{(\pi_r)}$, i.e., $\Ima{(\pi_r)}$ is a semialgebraic subset of $D_r$ with
\begin{equation}\label{eq:bigger}
\dim \Ima{(\pi_r)} = \dim D_r > \dim \partial D_r.
\end{equation}
The required result then follows from Theorem~\ref{thm:riden1} and Theorem~\ref{thm:mapopencell} with the trivial cell $D_r(I, J, K) \supseteq D_r \setminus \partial D_r$.
\end{proof}

A measure theoretic consequence of Corollary~\ref{cor:uniqueness} is that there is a positive measured subset of nonnegative tensors, such that each nonnegative tensor in this subset has a unique best nonnegative rank-$r$ approximation, and furthermore this approximation has a unique nonnegative rank-$r$ decomposition.

In the case of real tensors, it is possible that best rank-$r$ approximations always lie on the boundary of the set of tensors of rank $\le r$ \cite[Section~8]{DesiL08:simax}. So one might perhaps wonder whether Corollary~\ref{cor:uniqueness} is vacuous. Fortunately this is not the case for nonnnegative tensors provided that $r < r_g$ and $U \otimes V \otimes W$ is not $r$-defective. In fact, the condition \eqref{eq:bigger} implies that $\pi_r(A)$ is not always in $\partial D_r$.

For the special cases $r = 2$ and $3$, we can say considerably more than Corollary~\ref{cor:uniqueness}. We will first make an observation regarding the case when $\pi_r(A) \in \mathcal{L}$ where $\mathcal{L}$ is as defined in \eqref{eq:L}.
\begin{lemma}\label{le:supp}
Let $\pi_r(A) \in \mathcal{L}$. Then
\begin{align*}
%\label{eq:suppu}
\operatorname{supp}(u_1) \cup \cdots \cup \operatorname{supp}(u_r) &= \{1, \dots, n_U\}, \\
\operatorname{supp}(v_1) \cup \cdots \cup \operatorname{supp}(v_r) &= \{1, \dots, n_V\}, \\
\operatorname{supp}(w_1) \cup \cdots \cup \operatorname{supp}(w_r) &= \{1, \dots, n_W\}.
\end{align*}
\end{lemma}
\begin{proof}
Suppose $1 \notin \bigcup_{i=1}^r \operatorname{supp}(u_i)$. Then by definition
\[
\langle A - \pi_r(A), e_1 \otimes v_1 \otimes w_1 \rangle \le 0
\]
where $e_1 = (1, 0, \dots, 0)$. Since the coordinate $(\pi_r(A))_{1jk} = 0$ for any $j =1, \dots, n_V$, $k = 1, \dots, n_W$, and $A$ is positive, we have that $(A - \pi_r(A))_{1jk} > 0$. On the other hand, $(e_1 \otimes v_1 \otimes w_1)_{ijk} = 0$ for $i \neq 1$, and $(e_1 \otimes v_1 \otimes w_1)_{1jk} \ge 0$. Hence
\[
\langle A - \pi_r(A), e_1 \otimes v_1 \otimes w_1 \rangle > 0,
\]
a contradiction.
\end{proof}

A cell $D_r(I, J, K) $ is called \emph{admissible} if
\[
\bigcap\nolimits_{i=1}^r I_i = \bigcap\nolimits_{i=1}^r J_i = \bigcap\nolimits_{i=1}^r K_i = \varnothing.
\]
By Proposition~\ref{prop:boundary}, Theorem~\ref{thm:mapopencell}, and Lemma~\ref{le:supp}, if $A \in \mathcal{N}$, then there is an open neighborhood $\mathcal{V}$ of $A$ such that $\pi_r(\mathcal{V})$ contains an open subset of some admissible cell $D_r(I, J, K) $. For small values of $r$, we may check these admissible cells and possibly obtain uniqueness for nonnegative rank-$r$ decomposition of $\pi_r(A)$ for a general $A$. We will do this explicitly for $r =2$ and $3$.
\begin{theorem}\label{thm:uni23}
Let $r =2$ or $3$ and let $n_U , n_V, n_W \ge 3$. Then for a general $A \in U_+ \otimes V_+ \otimes W_+$, its unique best nonnegative rank-$r$ approximation $\pi_r(A)$ has a unique nonnegative rank-$r$ decomposition.
\end{theorem}

\begin{proof}
By Corollary~\ref{cor:uniqueness}, it remains to check the case $\pi_r(A) \in \partial D_r$ for a general $A$. Theorem~\ref{thm:mapopencell} and Lemma~\ref{le:supp} further restrict the remaining case to checking (i) whether $\pi_r(A)$ can be contained in an admissible cell, and (ii) whether $\pi_r(A)$ contained in an admissible cell (if any) has a unique decomposition.

When $r = 2$, for a general $p$ in any admissible cell $D_r(I, J, K) $, let $p =  u_1 \otimes v_1 \otimes w_1 + u_2 \otimes v_2 \otimes w_2$ be its nonnegative rank-$2$ decomposition. Then each set $\{ u_1, u_2 \}$, $\{v_1, v_2 \}$, and $\{w_1, w_2 \}$ consists of a pair of linearly independent vectors. By \cite{Krus77:laa}, $p$ has a unique real rank-$2$ decomposition and thus the nonnegative rank-$2$ decomposition is unique.

When $r = 3$, we may assume without loss of generality \cite[Theorem~5.2]{DesiL08:simax} that $n_U = n_V = n_W = 3$. The only situation where a general point $p$ of an admissible cell $D_r(I, J, K) $ does not have a unique nonnegative rank-$r$ decomposition is if
\begin{gather*}
I_1 = I_2 = \{ 2, 3 \}, \; I_3 \subseteq \{ 1 \}, \qquad J_1 = J_3 = \{ 2, 3 \}, \; J_2 \subseteq \{ 1 \}, \\
K_2 = K_3 = \{ 2, 3 \}, \;  K_1 \subseteq \{ 1 \},
\end{gather*}
up to a permutation of the index set $\{ 1, 2, 3 \}$. We claim that $\pi_r(A) $ cannot be contained in such a cell $D_r(I, J, K) $. Suppose not and  $\pi_r(A) \in D_r(I, J, K)$, i.e.,
\[
u_1 = u_2 = (1, 0, \dots, 0), \quad v_1 = v_3 = (1, 0, \dots, 0), \quad w_2 = w_3 = (1, 0, \dots, 0).
\]
Then $(\pi_r(A))_{1jk} = 0$ for $j = 2, 3$, $k  = 2, 3$. Let
\[
p = u_1 \otimes v_1 \otimes w_1 + u_2 \otimes v_2 \otimes (w_2 + z) + u_3 \otimes v_3 \otimes w_3
\]
for some $z = (0, \alpha, \beta)$ with $\alpha, \beta > 0$ small enough. Then $\norm{A - p} < \norm{A - \pi_r(A)}$ for a positive $A$, contradicting the definition of $\pi_r(A)$. Therefore $\pi_r(A) \notin D_r(I, J, K) $, a contradiction.
\end{proof}

It is possible that a general point in an admissible cell $D_r(I, J, K) $ may have non-unique nonnegative rank-$r$ decompositions. To show uniqueness, we need to exclude such a possibility, i.e.,  %determine if there is a nonempty open subset $\mathcal{V}$ such that a general $\pi_r(A) \in \pi_r(\mathcal{V})$ has a unique nonnegative rank-$r$ decomposition, we need to 
check whether $\pi_r(A)$ is contained in such a cell for a typical $A$. For small values of $r$, we may test all cells case-by-case but evidently this becomes prohibitive for even moderately large values of $r$. Further results in this direction would require more precise descriptions of $I_1, \dots, K_r$ where $D_r(I, J, K)  \cap \Ima \pi_r \neq \varnothing$.

\section*{Acknowledgment}
The authors would like to thank G.~Blekherman, L.~Chiantini, I.~Domanov, P.~Eyssidieux, S.~Friedland, J.~M.~Landsberg, B.~Mourrain, Z.~Teitler and N.~Vannieuwenhoven for useful discussions. The authors are very grateful to the anonymous referees for their suggestions and comments that greatly improved and clarified our manuscript.

%\bibliographystyle{siamplain}
%\bibliography{nonneg}

\bibliographystyle{siamplain}

%%%%%%%%%%%%%%%%%%%%%%%%%%%%%%%%%%%%%%%%%%%%%%%%%%%%%%%%%

\end{document}